\documentclass[10pt, a4paper]{article}
\usepackage[T1]{fontenc}
\usepackage[utf8x]{inputenc}
\usepackage{mathtools}

\let\vec\relax 
\DeclareMathAccent{\vec}{\mathord}{letters}{"7E} 
\usepackage{amsmath}
\usepackage{diagbox}
\usepackage{rotating}

\usepackage{pdfpages}
\usepackage{hyperref}
\usepackage{color}
\usepackage{makecell}
\usepackage{newtxtext,newtxmath}
\usepackage{cite}
\usepackage{cases} 
\usepackage{graphicx}
\usepackage[]{units}
\usepackage{enumitem}
\usepackage{booktabs}
\usepackage[toc,page]{appendix}
\usepackage{url}
\usepackage{hyperref}

\newtheorem{lemma}{Lemma}
\newtheorem{proof}{Proof}
\newtheorem{remark}{Remark}
\newtheorem{definition}{Definition}
\newtheorem{theorem}{Theorem}

\usepackage{xcolor}

\usepackage{color}
\colorlet{Mycolor1}{green!20!orange!80!}
\iftrue  
\newcommand{\comments}[1]{\footnote{\textcolor{blue}{\textit{#1}}}}
\newcommand{\AL}[1]{\textcolor{black}{#1}}
\newcommand{\NGU}[1]{\textcolor{black}{#1}}
\newcommand{\NGUNEW}[1]{\textcolor{black}{#1}}
\newcommand{\VS}[1]{\textcolor{black}{#1}}
\else
\newcommand{\comments}[1]{}
\newcommand{\AL}[1]{#1}
\newcommand{\NGU}[1]{#1}
\newcommand{\NGUNEW}[1]{#1}
\newcommand{\VS}[1]{#1}
\fi

\providecommand{\keywords}[1]
{
  \small	
  \textbf{\textit{Keywords -}} #1
}

\begin{document}

\title{Distributionally robust chance-constrained Markov decision processes \thanks{This research was supported by DST/CEFIPRA Project No. IFC/4117/DST-CNRS-5th call/2017-18/2 and CNRS Project No. AR/SB:2018-07-440.}}

\author{Hoang Nam Nguyen\thanks{University Paris Saclay, CNRS, CentraleSupelec, L2S, Bat Breguet, 3 rue Joliot Curie, 91190, Gif-sur-Yvette, France (\href{mailto:hoang-nam.nguyen3@centralesupelec.fr}{hoang-nam.nguyen3@centralesupelec.fr}).}
\and Abdel Lisser \thanks{University Paris Saclay, CNRS, CentraleSupelec, L2S, Bat Breguet, 3 rue Joliot Curie, 91190, Gif-sur-Yvette, France (\href{mailto:abdel.lisser@centralesupelec.fr}{abdel.lisser@centralesupelec.fr}), Corresponding Author}
\and Vikas Vikram Singh\thanks{Department of Mathematics, Indian Institute of Technology Delhi, Hauz Khas, New Delhi, 110016, India (\href{mailto:vikassingh@maths.iitd.ac.in}{vikassingh@maths.iitd.ac.in}).}}

\maketitle

\begin{abstract}
\VS{Markov decision process (MDP) is a decision making framework where a decision maker is interested in maximizing \AL{the} expected discounted value of a stream of rewards received at future stages at various states which are visited according to a controlled Markov chain. Many algorithms including linear programming methods are available in the literature to compute an optimal policy when the rewards and transition probabilities are deterministic.  In this paper, we consider \AL{an MDP} problem where the transition probabilities are known and the reward vector is a random  vector whose distribution is partially known. We formulate the MDP problem using distributionally robust chance-constrained optimization framework 
under various types of moments based uncertainty sets, and  statistical-distance based uncertainty sets defined using $\phi$-divergence and Wasserstein distance metric. For each type of uncertainty set, we consider the case where \AL{a} random reward vector has either a full support or a nonnegative support.  For the case of full support, we show that the distributionally robust chance-constrained Markov decision process is equivalent to a second-order cone programming problem for the moments and $\phi$-divergence distance based uncertainty sets, and it is equivalent to a mixed-integer second-order cone programming problem for an Wasserstein distance based uncertainty set. 
 For the case of nonnegative support, it is equivalent to a copositive  \AL{optimization} problem  and a biconvex optimization problem for the moments based uncertainty sets and Wasserstein distance based uncertainty set, respectively. As an application, we study  a machine replacement problem and illustrate numerical experiments on randomly generated instances.   
 }
\end{abstract}
\keywords{Markov decision processes, Distributionally robust chance-constrained optimization, Second-order cone programming, Copositive optimization, Mix-integer second-order cone programming, Biconvex optimization}
\section{Introduction}\label{Intro}
An MDP is a decision making framework to model the performance of a  stochastic system which evolves over  time according to a controlled Markov chain. We consider the case where the system has a finite number of states. At time $t=0$, the system is at some initial state $s_0\in S$, where $S$ is a finite state space, according to an initial distribution $\gamma$,  and a decision maker chooses an action $a_0\in A(s_0)$, where $A(s_0)$ denotes the set of finite number of actions available to the decision maker  at state $s_0$. As a consequence a reward $R(s_0,a_0)$ is earned and at time $t=1$, the system moves to a new state $s_1$ with probability $p(s_0,a_0,s_1)$. The same thing repeats at time $t=1$ and it continues for the infinite horizon. The decision taken at time $t$, which could be deterministic or randomized, may depend on the history $h_t$ at time $t$, where 
$h_t = (s_0, a_0, s_1, \ldots, s_{t-1},a_{t-1}, s_t)$. 
Let $H_t$ be the set of all possible histories at time $t$. A history dependent decision rule $f_t$ at time $t$ is defined as $f_t(h_t)\in \wp(A(s_t))$ for every history $h_t$ with final state $s_t$,  where  
$\wp(A(s_t))$ denotes the set of probability distributions on the action set $A(s_t)$. A sequence of history dependent decision rules $f^h=(f_t)_{t=0}^\infty$ is called a history dependent policy. The policy is called Markovian if 
each  $f_t$ in the sequence $(f_t)_{t=0}^\infty$ depends only on the state at time $t$.  A Markovian policy $(f_t)_{t=0}^\infty$  is called a stationary policy if there exists a decision rule $f$ such that $f_t=f$ for all $t$.
\VS{Therefore, a  stationary policy can be represented as a sequence of \AL{the} same decision rules $(f,f,\dots)$ and with some abuse of notations we can denote it as $f$, and  define $f=(f(s))_{s\in S}$ such that $f(s)\in \wp(A(s))$  for every $s\in S$.} As per a stationary policy $f$, whenever the Markov chain visits state $s$, the decision maker chooses an action $a$ with probability $f(s,a)$.
We denote the set of all history dependent and stationary policies by $PO_{HD}$ and $PO_S$, respectively. 
A history dependent policy $f^h \in PO_{HD}$ and \AL{an} initial distribution $\gamma$  define a probability measure $P_{\gamma}^{f^h}$ over the state and action trajectories, and $E_{\gamma}^{f^h}$ denotes the expectation operator corresponding to \AL{the} probability measure $P_{\gamma}^{f^h}$. 
For a given policy $f^h$
and an initial distribution $\gamma$, the expected discounted reward at a discount factor $\alpha\in(0,1)$ is defined as 
\cite{altman1999constrained, puterman1994discrete}
\begin{align}\label{exp-disc}
  V_{\alpha}(\gamma, f^h) &= \VS{(1-\alpha)} \mathbb{E}_{\gamma}^{f^h} \left( \sum_{t=0}^{\infty} \alpha^{t} R(X_t,A_t)\right), \nonumber\\
   &= \sum_{s\in S}\sum_{a\in A(s)} g_{\alpha}(\gamma,f^h;s,a)R(s,a),
\end{align}
where $X_t$ and $A_t$ represent \AL{the} state and \AL{the} action at time $t$, respectively. \VS{For a given policy $f^h$, the set} $\{g_{\alpha}(\gamma,f^h;s,a)\}_{(s,a)}$ is the occupation measure  defined by
\begin{align*}
    g_{\alpha}(\gamma,f^h;s,a)&= \VS{(1-\alpha)}\sum_{t=0}^\infty \alpha^{t} P_{\gamma}^{f^h}(\mathbb{X}_t=s,\mathbb{A}_t=a),\ \forall\ s\in S, \ a\in A(s).
    \end{align*}
When \AL{the} running rewards and \AL{the} transition probabilities are stationary, i.e., $R(X_t=s, A_t= a) = R(s,a)$ and $P(X_{t+1}=s'|X_t=s, A_t=a) = p(s, a, s')$ \VS{for all $t$}, we can restrict to stationary policies without loss of optimality   \cite{altman1999constrained, puterman1994discrete}.
It follows from Theorem 3.2 on p. 28 in \cite{altman1999constrained} that the set of occupation measures corresponding to history dependent policies is equal to the set of occupation measures corresponding to stationary policies
and further it is equal to the  set
\begin{gather*}
\mathcal{Q}_{\alpha}(\gamma)=\Bigg\{\rho\in \mathbb{R}^{|\mathcal{K}|} \ \bigl\lvert \, \sum_{(s,a)\in \mathcal{K}}\rho(s,a)\Big(\delta(s', s)
 -\alpha p(s,a, s')\Big) 
= (1-\alpha) \gamma(s'),\ \forall \ s'\in S, \\ \rho(s,a)\geq 0, \ \forall \ s\in S,\ a\in A(s)\Bigg\}, 
\end{gather*}
such that the \VS{value of} \AL{the} expected discounted reward defined by \eqref{exp-disc} remains the same; $\delta(s',s)$ is the Kronecker delta \NGU{and $\mathcal{K} = \left\{(s,a) \mid s \in S, a \in A(s)\right\}$}. \VS{Therefore, the optimal policy of the MDP problem can be obtained by solving the following linear programming problem \cite{puterman1994discrete}}
\begin{equation} \label{equi_lp}
     \max_{\rho\in \mathcal{Q}_{\alpha}(\gamma) } \rho^\text{T} R,
\end{equation}
where $R=(R(s,a))_{s\in S, a\in A(s)}$ is a running reward vector and $\text{T}$ denotes the transposition.  
If $\rho^*$ is an optimal solution of \eqref{equi_lp}, the stationary optimal policy $f^*$ can be defined as 
\[
f^*(s,a)= \dfrac{\rho^*(s,a)}{\sum_{a \in A(s)}\rho^*(s,a)}, \ \forall \ s\in S,\ a\in A(s),
\]
whenever the denominator is nonzero \big(if it is zero, we choose $f^*(s)$ arbitrarily from $\wp(A(s))$\big) \cite{altman1999constrained}. In practice, the MDP model parameters $R(\cdot)$ and $p(\cdot)$ are not known \AL{in advance} and are estimated from historical data. This leads to errors in the optimal policies  \cite{mannor2007bias}. Most efforts to take into account this uncertainty  focused on the study of robust MDPs where the rewards or the transition probabilities are known to belong to a prespecified uncertainty set \cite{iyengar2005robust, nilim2005robust, varagapriya2022constrained, white1994markov,   wiesemann2013robust}.
However, Delage and Mannor \cite{delage2010percentile} showed that the robust MDP approach usually leads to conservative policies. For this reason,  a chance-constrained Markov decision process (CCMDP) was introduced in \cite{delage2010percentile},  where the controller obtains the expected discounted reward with certain confidence. In \cite{delage2010percentile}, the case of random rewards and random transition probabilities are considered separately and it is shown that a CCMDP is equivalent to a second-order cone programming (SOCP) problem when \AL{the} running reward vector follows a multivariate normal distribution and \VS{the} transition probabilities are exactly known. When \VS{the} transition probabilities follow  Dirichlet distribution and  \VS{the} running rewards are exactly known, the CCMDP problem becomes intractable and \AL{the} optimal policies can be computed using approximation methods. \VS{Varagapriya et al. \cite{VagapriyaCMDP} considered a CMDP problem with joint chance constraint where \AL{the} running cost vectors are random vectors and \AL{the} transition probabilities are known.  They proposed two SOCP based approximations which give  upper and lower bounds to the CMDP problem if the cost vectors follow multivariate elliptical distributions and \AL{the} dependence among \AL{the} constraints is driven by a Gumbel-Hougaard copula.}
 
 \VS{ In many practical situations, it is often the case that only a partial information about the underlying distribution is available based on historical data. 
 In that case,  a distributionally robust approach, is used to model the uncertainties, which assumes  that the true distribution belongs to an uncertainty set based on its partially available information. 
 Such an approach has been used in modelling the uncertainties of many optimization and game problems \cite{jiang2016data,liu2022distributionally, singh2017distributionally}.
 There are \AL{at least} two popular ways to construct an uncertainty set for the distribution 
 of the uncertain parameters. The first one is based on the partial information on moments of the true distribution and the second one is based on the 
 statistical distance between the true distribution and a nominal distribution.
 The moments-based uncertainty sets assume certain conditions on the first two moments \cite{cheng2014distributionally,delage2010distributionally, popescu2007robust}.
The statistical distance-based uncertainty sets  contain all the distributions which lie inside a ball of small radius and center at a nominal distribution which is usually considered to be an empirical distribution or a normal  distribution \cite{esfahani2018data, jiang2016data}. To define a distance between the distributions, either a $\phi-$divergence \cite{ben2013robust, jiang2016data} or Wasserstein distance metric is used
\cite{esfahani2018data, gao2016distributionally, zhao2018data}.}

In this paper, we consider an infinite horizon MDP with discounted payoff criterion defined in Section \ref{Intro} where the reward vector is a random vector and the transition probabilities are known. 
 The distribution of the reward vector is not completely known and it is assumed to belong to a given uncertainty set.
 We formulate the random discounted reward \AL{with} a distributionally robust chance constraint which  guarantees the maximum reward for a given policy with at least a given level of confidence. 
 We call this class of MDP as a distributionally robust chance-constrained Markov decision process (DRCCMDP). 
 The random reward vector has either a  full support or a nonnegative support.    We consider both moments and statistical distance based uncertainty sets. The main contributions of the paper are  as follows.
 \begin{enumerate}
     \item We consider three different types of moments based uncertainty sets based on the full/partial information on the first two moments of the random reward vector. For the case of full support and nonnegative support, a DRCCMDP problem is equivalent to an SOCP problem and a copositive \AL{optimization} problem, respectively.
     \item We consider four different types of $\phi$-divergences to construct statistical distance based uncertainty sets. We show that a DRCCMDP problem is equivalent to an SOCP problem when the nominal distribution is a normal distribution.
     \item We consider the nominal distribution to be an empirical distribution when  statistical distance based uncertainty set is defined \AL{with} Wasserstein distance metric. For the case of full support and nonnegative support, we show that a DRCCMDP problem is equivalent to a mixed integer second-order cone programming (MISOCP) problem and a biconvex optimization problem, respectively. 
     \item We illustrate our theoretical results on a machine replacement problem \cite{delage2010percentile}.
 \end{enumerate}
\VS{
The paper is organized as follows. In Section \ref{DRCCMDP-section}, we define a DRCCMDP under a discounted payoff criterion. Section \ref{moment-based} contains a DRCCMDP under moments based uncertainty sets and \AL{their} equivalent reformulations for the case of full and nonnegative supports. A DRCCMDP under statistical distance based uncertainty sets defined using $\phi$-divergence metric and Wasserstein distance metric and \AL{their} equivalent reformulations are presented in Section \ref{stat-dist-un}.
The numerical results on a machine replacement problem is given in Section \ref{machine_rep}. We conclude the paper in Section \ref{conclusion}.
}


\section{Distributionally robust chance constrained Markov decision process}\label{DRCCMDP-section}

We consider an infinite horizon MDP defined in Section \ref{Intro} where the transition probabilities are exactly known and the running reward vector is a random vector defined on a probability space $(\Omega, \mathcal{F}, \mathbb{P})$ \VS{which is denoted as $\hat{R}$. Therefore, for each realization $\omega\in \Omega$, $\hat{R}(s,a, \omega)$ represents a real valued reward received at state $s$ when an action $a$ is taken.}  
We assume that the random vector $\hat{R}$ does not vary with time. Since $\hat{R}$ is a random vector,  for a given policy $f^h$ and initial distribution $\gamma$, the expected discounted 
reward defined by \eqref{exp-disc} becomes a random variable. We consider the case where the controller is interested in a maximum discounted reward which can be obtained with at least a given confidence level $(1-\epsilon)$, where $\epsilon \in (0,1)$. This leads to the following CCMDP problem  
\begin{align}\label{CCMDP}
     &\sup_{y \in \mathbb{R},\ f^h \in F_{HD}} y \nonumber\\
     \mbox{s.t.} \,\,\,\,\, &\mathbb{P}\left(V_\alpha(s, f^h) \geq y \right) \geq 1-\epsilon. 
\end{align}
Since the transition probabilities are exactly known, it follows from the discussion given in Section \ref{Intro} that we can represent the CCMDP problem \eqref{CCMDP} equivalently in terms of decision vector $(y, \rho)$ as follows
\begin{align}\label{CCMDP-equi}
     &\sup\,\,\, y \nonumber\\
     \mbox{s.t.} \,\,\,\,\, & (\text{i}) \quad \mathbb{P}\left(\rho^\text{T} \hat{R} \ge y \right) \ge 1-\epsilon,\nonumber\\
     & (\text{ii}) \quad \rho\in \mathcal{Q}_\alpha(\gamma).
\end{align}
If then vector $\hat{R}$ follows a multivariate normal distribution, the optimization problem \eqref{CCMDP-equi} is equivalent to an SOCP problem \cite{delage2010percentile}. The above result can be generalized for elliptically symmetric distributions because the linear chance constraint $(\text{i})$ present in \eqref{CCMDP-equi} is equivalent to a second order cone constraint \cite{Henrionpaper}.

However, in most of the practical situations, we only have partial information about the underlying probability distributions. Such situations can be handled \AL{with the}  distributionally robust optimization approach, i.e., we assume that the distribution of $\hat{R}$ belongs to an uncertainty set. This leads to the following DRCCMDP problem 
\begin{align}\label{DRCCMDP-opt-ex}
     &\sup\,\,\, y \nonumber\\
     \mbox{s.t.} \,\,\,\,\, & (\text{i})\quad \inf_{F \in \mathcal{D}}\mathbb{P}\AL{_F}\left(\rho^\text{T} \hat{R} \geq y \right) \geq 1-\epsilon,\nonumber\\
     & (\text{ii})\quad \rho\in \mathcal{Q}_\alpha(\gamma),
\end{align}
where  $F$ is the distribution of $\hat{R}$ and  $\mathcal{D}$ is a given uncertainty set. \VS{The first constraint of \eqref{DRCCMDP-opt-ex} can be written as 
\[
\sup_{F \in \mathcal{D} }\mathbb{P}\AL{_F}\left(\rho^\text{T} \hat{R} < y \right) \le \epsilon.
\]
}
Note that $\mathbb{P}\AL{_F} (\rho^\text{T} \hat{R} \leq y-\theta) \leq \mathbb{P}\AL{_F} (\rho^\text{T} \hat{R} < y) \leq \mathbb{P}\AL{_F} (\rho^\text{T} \hat{R} \leq y)$ for every $\theta > 0$. \VS{Therefore, we can replace
$\sup_{F \in \mathcal{D} }\mathbb{P}\AL{_F}\left(\rho^\text{T} \hat{R} < y \right)$ by $\sup_{F \in \mathcal{D} }\mathbb{P}\AL{_F}\left(\rho^\text{T} \hat{R} \le y \right)$.}
Then, the problem \eqref{DRCCMDP-opt-ex} is equivalent to the following problem
\begin{align}\label{DRCCMDP-opt}
    &\sup \quad y \nonumber\\
     \mbox{s.t.} \,\,\,\,\, & (\text{i})\quad \sup_{F \in \mathcal{D} }\mathbb{P}\AL{_F}\left(\rho^\text{T} \hat{R} \leq y \right) \leq \epsilon,\nonumber\\
     & (\text{ii})\quad  \rho\in\mathcal{Q}_{\alpha}(\gamma).
\end{align}
\VS{
In the following sections, we study different types of uncertainty sets of $\hat{R}$ which are defined  using i) partial information of moments of $\hat{R}$, ii) $\phi$-divergence distance, and iii) Wasserstein distance.
 For each uncertainty set, we consider the cases of full and nonnegative supports of $\hat{R}$. We derive equivalent reformulations of DRCCMDP problem \eqref{DRCCMDP-opt-ex} (or \eqref{DRCCMDP-opt} equivalently) for each uncertainty set.
}

\section{Moments based uncertainty sets} \label{moment-based}
\VS{Let $\mu\in \mathbb{R}^{\mathcal{|K|}}$ \AL{the mean vector} and $\Sigma\succ 0$ a $\mathcal{|K|}\times \mathcal{|K|}$ positive definite matrix.} 
\NGU{We consider 3 types of moments based uncertainty sets of the distribution of $\hat{R}$ defined as follows:}
\VS{
\begin{enumerate}
    \item 
\textbf{Uncertainty set with known mean and known covariance matrix:}
The uncertainty set of the distribution of $\hat{R}$ in this case is defined by 
\begin{equation}\label{definition D1}
\mathcal{D}_1(\varphi,\mu,\Sigma) = \left\{F \in \mathcal{M}^+ \left|
\begin{array}{l}
\mathbb{E}(\mathbf{1}_{\left\{\hat{R}\in \varphi\right\}})=1,\\
\mathbb{E}(\hat{R}) = \mu,\\
\mathbb{E}[(\hat{R}-\mu)(\hat{R}-\mu)^{\text{T}}] = \Sigma.
\end{array}
\right.
\right\},
\end{equation}
\item 
\textbf{Uncertainty set with known mean and unknown covariance matrix:}
The uncertainty set of the distribution of $\hat{R}$ in this case is defined by
\begin{equation}\label{Uncertainty set known unknown}
\mathcal{D}_2(\varphi,\mu,\Sigma,\delta_0) = \left\{F \in \mathcal{M}^+ \left|
\begin{array}{l}
\mathbb{E}(\mathbf{1}_{\left\{\hat{R} \in \varphi\right\}}) = 1,\\
\mathbb{E}(\hat{R})=\mu,\\
\mathbb{E}[(\hat{R}-\mu)(\hat{R}-\mu)^{\text{T}}] \preceq \delta_0 \Sigma.
\end{array}
\right.
\right\},
\end{equation}
\item 
\textbf{Uncertainty set with unknown mean and unknown covariance matrix:}
The uncertainty set of the distribution of $\hat{R}$ in this case is defined by
\begin{equation}\label{Uncertainty set unknown unknown}
\mathcal{D}_3(\varphi,\mu,\Sigma,\delta_1,\delta_2) = \left\{F \in \mathcal{M}^+ \left|
\begin{array}{ll}
&\mathbb{E}(\mathbf{1}_{\left\{\hat{R} \in \varphi\right\}}) = 1,\\
&[\mathbb{E}(\hat{R})-\mu]^{\text{T}}\Sigma^{-1}[\mathbb{E}(\hat{R})-\mu] \leq \delta_1,\\
&\mathbb{E}[(\hat{R}-\mu)(\hat{R}-\mu)^{\text{T}}] \preceq \delta_2 \Sigma.
\end{array}
\right.
\right\},
\end{equation}
\end{enumerate}
where $\varphi\subset \mathbb{R}^{|\mathcal{K}|}$ is the support of $\hat{R}$ which we assume to be a convex set, $\mathcal{M}^+$ is the set of all probability measures on $\mathbb{R}^{|\mathcal{K}|}$ with Borel $\sigma-$algebra,    
$\delta_1 \geq 0,\delta_2,  \delta_0 \geq 1$, $\mu \in \text{RI}(\varphi)$; $\text{RI}(\varphi)$ denotes the relative interior of $\varphi$. The notation 
$A \preceq B$ implies that $B-A$ is a  positive semidefinite matrix and $\mathbf{1}_{\left\{\cdot\right\}}$ denotes the indicator function.
}
\subsection{DRCCMDP with moments based uncertainty sets under full support}
\VS{We consider the case when the random vector $\hat{R}$ has full support, i.e., $\varphi = \mathbb{R}^{|\mathcal{K}|}$. We show that the DRCCMDP problem is equivalent to an SOCP problem.}
\begin{theorem}
Consider the DRCCMDP problem \eqref{DRCCMDP-opt-ex} where the distribution of $\hat{R}$ belongs to the uncertainty sets defined by \eqref{definition D1}, \eqref{Uncertainty set known unknown}, \eqref{Uncertainty set unknown unknown}, and the support $\varphi=\mathbb{R}^{|\mathcal{K}|}$. Then, the \emph{DRCCMDP} \eqref{DRCCMDP-opt-ex} can be reformulated equivalently as the following \emph{SOCP}
\begin{align}\label{SOCP known known real full support}
    &\max\,\,\,\, y \nonumber\\
    \mbox{s.t.} \,\,\,\,\,
    & (\emph{i})\quad \mu^{\emph{T}} \rho - \kappa\|\Sigma^{\frac{1}{2}}\rho\|_2 \geq y,\nonumber\\
    & (\emph{ii}) \quad \rho\in\mathcal{Q}_{\alpha}(\gamma),
    \end{align}
where $||\cdot||_2$ denotes the Euclidean norm  and $\kappa$ is a real number whose value for each uncertainty set is given 
 in Table \ref{table of kappa}.
\begin{table}[ht]\small 
\centering
\caption{ Value of $\kappa$ for moments based uncertainty set} \vspace{.2cm}
\begin{tabular}{|*{5}{c|}}
\hline
Uncertainty set & $\mathcal{D}=\mathcal{D}_1(\varphi,\mu,\Sigma)$ & $\mathcal{D}=\mathcal{D}_2(\varphi,\mu,\Sigma,\delta_0)$ & $\mathcal{D}=\mathcal{D}_3(\varphi,\mu,\Sigma,\delta_1,\delta_2)$\\
\hline
$\kappa$ & $\sqrt{\frac{1-\epsilon}{\epsilon}}$ & $\sqrt{\frac{(1-\epsilon)\VS{\delta_0}}{\epsilon}}$ & $\sqrt{\frac{(1-\epsilon)\delta_2}{\epsilon}}+\sqrt{\delta_1}$\\
\hline
\end{tabular}\label{table of kappa}
\end{table}
\end{theorem}
\begin{proof}
\VS{The proof follows from the fact that for each uncertainty set the distributionally robust chance constraint $(\text{i})$ of \eqref{DRCCMDP-opt-ex}  is equivalent to a second-order cone constraint. The uncertainty set \eqref{definition D1} has been widely studied in the literature \cite{ElGhaouiJOTA, ghaoui2003worst}. For the uncertainty sets \eqref{Uncertainty set known unknown} and \eqref{Uncertainty set unknown unknown}, it can be proved using similar arguments used in Lemma 3.1 and 
Lemma 3.2 of \cite{HNConference2022} which are based on the one-sided Chebyshev inequality \cite{liu2022distributionally}}. 
\end{proof}

\subsection{DRCCMDP with moments based uncertainty sets under nonnegative support}
We consider the case where the support of \AL{the} random vector $\hat{R}$ is a nonnegative orthant of $|\mathcal{K}|$-dimensional Euclidean space, i.e., $\varphi=\mathbb{R}_+^{|\mathcal{K}|}$. We show that the DRCCMDP problem \eqref{DRCCMDP-opt} is equivalent to a copositive \AL{optimization} problem.
\begin{theorem}\label{theorem nonnegative moment}
Consider a DRCCMDP problem \eqref{DRCCMDP-opt} with $\varphi=\mathbb{R}_+^{|\mathcal{K}|}$. Then, the following results hold.
\begin{enumerate}
    \item If the distribution of $\hat{R}$ belongs to the uncertainty set defined by \eqref{definition D1},  the DRCCMDP problem \eqref{DRCCMDP-opt} is equivalent to the following copositive \AL{optimization} problem
    \begin{align}\label{matrix form unbounded nonegative polytopic}
    &\max\,\,\,\,\,y \nonumber\\
    \mbox{s.t.} \,\,\,\,\,
    & (\emph{i}) \quad -t-Q \circ \Sigma - q^{\emph{T}} \mu \leq \epsilon,\nonumber \\
    & (\emph{ii}) \quad \begin{pmatrix} -Q & \vline & -\frac{1}{2}q+Q\mu\\
    \hline
    -\frac{1}{2}q^{\emph{T}}+\mu^{\emph{T}}Q &\vline & -t - \mu^{\emph{T}} Q \mu \end{pmatrix} \in \emph{COP}^{|\mathcal{K}|  + 1},\nonumber\\
    & (\emph{iii}) \begin{pmatrix} -Q & \vline & -\frac{1}{2}q+Q\mu + \lambda \rho\\
    \hline
    -\frac{1}{2}q^{\emph{T}}+\mu^{\emph{T}}Q + \lambda \rho^{\emph{T}}&\vline & -t - \mu^{\emph{T}} Q \mu -1 - \lambda y \end{pmatrix} \in \emph{COP}^{|\mathcal{K}|  + 1},\nonumber\\
    & (\emph{iv}) \quad  Q\in \mathcal{S}^{|\mathcal{K}|}, \lambda \geq 0, \nonumber\\
    & (\emph{v}) \quad \rho\in\mathcal{Q}_{\alpha}(\gamma).
    \end{align}
    \item  If the distribution of $\hat{R}$ belongs to the uncertainty set defined by \eqref{Uncertainty set known unknown},  the DRCCMDP problem \eqref{DRCCMDP-opt} is equivalent to the following copositive \NGUNEW{optimization} problem
    \begin{align}\label{matrix form unbounded nonegative known unknown}
    &\max\,\,\,\,\,y \nonumber\\
    \mbox{s.t.} \,\,\,\,\,
    & (\emph{i}) \quad -t-\mu^{\emph{T}}q-\mu^{\emph{T}}Q\mu+\delta_0\Sigma \circ Q \leq \epsilon,\nonumber \\
    & (\emph{ii}) \quad \begin{pmatrix}Q&\vline&-\frac{1}{2}q-Q\mu\\\hline-\frac{1}{2}q^{\emph{T}}-\mu^{\emph{T}}Q&\vline&-t\end{pmatrix} \in \emph{COP}^{|\mathcal{K}|  + 1},\nonumber\\
    & (\emph{iii}) \quad \begin{pmatrix}Q&\vline&\frac{1}{2}(-q+\lambda \rho)-Q\mu\\\hline\frac{1}{2}(-q+ \lambda \rho)^{\emph{T}}-\mu^{\emph{T}}Q&\vline&-t-1-\lambda y\end{pmatrix} \in \emph{COP}^{|\mathcal{K}|  + 1},\nonumber\\
    & (\emph{iv}) \quad   Q \in \mathcal{S}_+^{|\mathcal{K}|}, \lambda \geq 0, \nonumber\\
    & (\emph{v}) \quad \rho\in\mathcal{Q}_{\alpha}(\gamma).
    \end{align}
    \item If the distribution of $\hat{R}$ belongs to the uncertainty set defined by \eqref{Uncertainty set unknown unknown},  the DRCCMDP problem \eqref{DRCCMDP-opt} is equivalent to the following copositive \NGUNEW{optimization} problem
    \begin{align}\label{matrix form unbounded nonegative unknown unknown}
    &\max\,\,\,\,\,y \nonumber\\
    \mbox{s.t.} \,\,\,\,\,
    & (\emph{i}) \quad r+t \leq \epsilon,\nonumber \\
    & (\emph{ii}) \quad \begin{pmatrix}Q&\vline&\frac{1}{2}q\\\hline\frac{1}{2}q^{\emph{T}}&\vline&r\end{pmatrix} \in \emph{COP}^{|\mathcal{K}|  + 1},\nonumber\\
    & (\emph{iii}) \quad t \geq (\delta_2 \Sigma + \mu \rho^{\emph{T}}) \circ Q + \rho^{\emph{T}} q + \sqrt{\delta_1} ||\Sigma^{\frac{1}{2}}(q+2Q\mu)||_2, \nonumber\\ 
    & (\emph{iv}) \quad \begin{pmatrix}Q&\vline&\frac{1}{2}(q+ \lambda \rho)\\\hline\frac{1}{2}(q+ \lambda \rho)^{\emph{T}}&\vline&r-1-\lambda y\end{pmatrix} \in \emph{COP}^{|\mathcal{K}|  + 1},\nonumber\\
    & (\emph{v}) \quad   Q \in \mathcal{S}_+^{|\mathcal{K}|}, \lambda \geq 0,\nonumber\\
    & (\emph{vi}) \quad \rho\in\mathcal{Q}_{\alpha}(\gamma),
    \end{align}
where \quad $\emph{COP}^{|\mathcal{K}| + 1} = \left\{M \in \mathcal{S}^{|\mathcal{K}|+1}\,\,|\,\, x^{\emph{T}} M x \geq 0, \ \forall\ x \in \mathbb{R}_+^{|\mathcal{K}|+1} \right\}$, 
$\mathcal{S}^{n}$ is the set of all real symmetric matrix of size $n \times n,$ 
$\mathcal{S}_+^{n}$ is the set of positive semidefinite \AL{matrices} of size $n \times n$,
$\circ$ denotes the Frobenius inner product and $\begin{pmatrix}&\vline&\\\hline&\vline&\end{pmatrix}$ denotes a block matrix (or a partitioned matrix).
\end{enumerate}
\end{theorem}

In order to prove the first result of Theorem \ref{theorem nonnegative moment}, we need the following lemma.

\begin{lemma}\label{lemma strong duality nonnegative polytopic}
\VS{Consider an optimization problem \begin{align}\label{definition vP vP}
    \sup_{F\in \mathcal{D}_1(\varphi,\mu,\Sigma)}
    \mathbb{P}_{\AL{F}}(\rho^{\emph{T}} \hat{R} \leq y),
\end{align}
where $\varphi=\mathbb{R}_+^{|\mathcal{K}|}$.
If the feasible set of \eqref{definition vP vP} is non-empty, the dual of \eqref{definition vP vP} is given by}
\begin{align*} 
    &\inf\quad -t-Q \circ \Sigma - q^{\text{T}} \mu \,\,\, \nonumber\\
    \mbox{s.t.}\,\,\, & (\emph{i}) \quad  \mathbf{1}_{\left\{\rho^{\text{T}} \xi \leq y\right\}} + q^{\text{T}} \xi + \xi^{\text{T}} Q \xi - 2\xi^{\text{T}} Q \mu + \mu^{\text{T}} Q \mu + t \leq 0,\ \forall\ \xi \in \mathbb{R}_{+}^{|\mathcal{K}|},\nonumber\\
    & (\emph{ii}) \quad Q\in \mathcal{S}^{|\mathcal{K}|},
\end{align*}
such that strong duality holds.
\end{lemma}

The proof is given in  Appendix \ref{Appendix Lemma non negative polytopic}.

\begin{proof}[Proof of Theorem \ref{theorem nonnegative moment}]
\begin{enumerate}
\item Let the distribution of $\hat{R}$ belongs to \AL{the} uncertainty set $\mathcal{D}_1(\phi,\mu,\Sigma)$. \VS{Using Lemma \ref{lemma strong duality nonnegative polytopic}, the optimization problem \eqref{DRCCMDP-opt} is equivalent to the following problem}
\begin{align}\label{SUB reformulated SDP unbounded known known nonnegative}
&\sup\quad y \nonumber\\
\mbox{s.t.} \,\, & (\text{i}) \quad -t-Q \circ \Sigma - q^{\text{T}} \mu \leq \epsilon,\nonumber\\
& (\text{ii}) \quad q^{\text{T}} \xi + \xi^{\text{T}} Q \xi - 2\xi^{\text{T}} Q \mu + \mu^{\text{T}} Q \mu + t \leq 0,\ \forall\ \xi \in \mathbb{R}_{+}^{|\mathcal{K}|},\nonumber\\
& (\text{iii}) \quad 1 + q^{\text{T}} \xi + \xi^{\text{T}} Q \xi - 2\xi^{\text{T}} Q \mu + \mu^{\text{T}} Q \mu + t \leq 0,\ \forall\ \xi \in \mathbb{R}_{+}^{|\mathcal{K}|},\, \rho^{\text{T}} \xi \leq y,\nonumber\\
& (\text{iv}) \quad Q\in \mathcal{S}^{|\mathcal{K}|}, \rho\in\mathcal{Q}_{\alpha}(\gamma).\nonumber\\
&
\end{align}
The constraint $(\text{ii})$ of \eqref{SUB reformulated SDP unbounded known known nonnegative} is equivalent to:
\begin{align*}
    (\xi^{\text{T}},1) U (\xi^{\text{T}},1)^{\text{T}} \geq 0, \ \forall\ \xi \in \mathbb{R}_{+}^{|\mathcal{K}|},
\end{align*}
where $U \in \mathcal{S}^{|\mathcal{K}|+1}$ such that \begin{align*}
    U = \begin{pmatrix} -Q & \vline & -\frac{1}{2}q+Q\mu\\
    \hline
    -\frac{1}{2}q^{\text{T}}+\mu^{\text{T}}Q &\vline & -t - \mu^{\text{T}} Q \mu \end{pmatrix}.
\end{align*}
Here, $(\xi^{\text{T}},1)$ denotes the row vector of size $1 \times (|\mathcal{K}|+1)$ with the last component equals 1 and the first $|\mathcal{K}|$ components are the components of $\xi$. The above inequality can be rewritten as 
\begin{align*}
    x^{\text{T}}Ux \geq 0, \ \forall\ x \in \mathbb{R}_{+}^{|\mathcal{K}| + 1},||x||_2=1.
\end{align*}
Using Proposition 5.1 in \cite{hiriart2010variational}, we deduce that the constraint $(\text{ii})$ of \eqref{SUB reformulated SDP unbounded known known nonnegative} is equivalent to $U \in \text{COP}^{|\mathcal{K}|+1}$. The constraint $(\text{iii})$ of \eqref{SUB reformulated SDP unbounded known known nonnegative} is equivalent to:
\begin{align}\label{constraint iii}
    -1 + (\xi^{\text{T}},1) U (\xi^{\text{T}},1)^{\text{T}} \geq 0, \ \forall \ \xi \in \mathbb{R}_{+}^{|\mathcal{K}|},\ \rho^{\text{T}} \xi \leq y.
\end{align}
Define,
\begin{equation}\label{minmax problem known known}
    \left\{
\begin{aligned}
    &s_{\text{P}} = \min_{\xi \in \mathbb{R}_{+}^{|\mathcal{K}|}} \max_{\lambda \geq 0}\,\,\,\,\,\,\, \mathcal{L}(\lambda,\xi, U, \rho,y).\\
    &s_{\text{D}} = \max_{\lambda \geq 0}\min_{\xi \in \mathbb{R}_{+}^{|\mathcal{K}|}}\,\,\,\,\,\,\, \mathcal{L}(\lambda,\xi, U, \rho,y).
    \end{aligned}
    \right.
\end{equation}
\NGU{where $\mathcal{L}(\lambda,\xi, U, \rho, y) = -1 + (\xi^{\text{T}},1) U (\xi^{\text{T}},1)^{\text{T}} + \lambda(\rho^{\text{T}} \xi - y)$.}
In \cite{cheng2014distributionally}, the authors use the Sion's minimax theorem \cite{sion1958general} to interchange the minimum and the maximum. However, since $\varphi$ is not compact, we cannot apply  the Sion's minimax theorem directly in this case. \VS{We show that $\varphi$ can be restricted to a compact set without loss of optimality.} \VS{For a given $U$ and $\rho$, we have}
\begin{align}\label{s_p finite}
    &s_{\text{P}} \leq \max_{\lambda \geq 0} \mathcal{L}(\lambda,0, U, \rho, y) \nonumber \\
    &= \max_{\lambda \geq 0}( -t-\mu^{\text{T}} Q \mu-\lambda y - 1) = -t - \mu^{\text{T}} Q \mu- 1< \infty
\end{align}
Therefore, using the min-max inequality $s_\text{D} \leq s_\text{P} < \infty$. Let $U_i = U + \frac{1}{2^i} \mathbf{I}_{|\mathcal{K}|+1}$ and $\rho_i = \rho + \frac{1}{2^i} \NGUNEW{\mathbf{1}}$, for every $i \in \mathbb{N}$, where $\mathbf{I}_{|\mathcal{K}|+1}$ denotes the identity matrix of size $|\mathcal{K}|+1$, \NGUNEW{$\mathbf{1}$} denotes the vector with all components equal to $1$. It is clear from the construction that $\rho_i > 0$ componentwise.
{Since, $\mathcal{L}$ is a continuous function w.r.t $U$ and $\rho$,  we have}
\begin{align*}
    \mathcal{L}(\lambda,\xi,U_i,\rho_i,y) \xrightarrow{i \rightarrow \infty} \mathcal{L}(\lambda,\xi,U,\rho,y), \ \forall\ \xi \in \mathbb{R}_+^{|\mathcal{K}|},\lambda \geq 0.
\end{align*}
Since, the min and max operators preserve the continuity, we have
\begin{align*}
    & \min_{\xi \in \mathbb{R}_+^{|\mathcal{K}|}} \max_{\lambda \geq 0}\,\,\,\,\,\,\, \mathcal{L}(\lambda,\xi,U_i,\rho_i,y) \xrightarrow{i \rightarrow \infty}\min_{\xi \in \mathbb{R}_+^{|\mathcal{K}|}} \max_{\lambda \geq 0}\,\,\,\,\,\,\,\mathcal{L}(\lambda,\xi,U,\rho,y).\\
    &\max_{\lambda \geq 0}\min_{\xi \in \mathbb{R}_+^{|\mathcal{K}|}}\,\,\,\,\,\,\, \mathcal{L}(\lambda,\xi,U_i,\rho_i,y) \xrightarrow{i \rightarrow \infty}\max_{\lambda \geq 0}\min_{\xi \in \mathbb{R}_+^{|\mathcal{K}|}}\,\,\,\,\,\,\,\mathcal{L}(\lambda,\xi,U,\rho,y).
\end{align*}
\NGU{This implies that, if $s_{\text{P}}=s_{\text{D}}$ holds for any $U_i,\rho_i$, $i \in \mathbb{N}$, it also holds for $U,\rho$.} \VS{For an arbitrary $U_i$ and $\rho_i$}, \NGU{let the the optimal solutions of minimax and maximin problems defined by \eqref{minmax problem known known} are $(\xi_\text{P},\lambda_\text{P})$ and $(\xi_\text{D},\lambda_\text{D})$, respectively. We prove that $\xi_\text{P}$ and $\xi_\text{D}$ are bounded, i.e., there exists $\Upsilon_\text{P} >0$ and $\Upsilon_\text{D} >0$ depending on $U_i,\rho_i$ and $y$ such that $||\xi_\text{P}||_2 \leq \Upsilon_\text{P}$ and $||\xi_\text{D}||_2 \leq \Upsilon_\text{D}$.}
\NGU{It is clear that $\lambda_\text{P} = 0$ and $\rho_i^{\text{T}} \xi_\text{P} - y \leq 0$. Hence, we have}
\begin{align*}
    s_\text{P} &= -1 + (\xi_\text{P}^{\text{T}},1) U_i (\xi_\text{P}^{\text{T}},1)^{\text{T}},\\
    &=-1 + (\xi_\text{P}^{\text{T}},1) U (\xi_\text{P}^{\text{T}},1)^{\text{T}} + \frac{1}{2^i}||\xi_\text{P}||_2^2 + \frac{1}{2^i}.
\end{align*}
\VS{From constraint $(\text{ii})$ of \eqref{SUB reformulated SDP unbounded known known nonnegative}, it follows that $(\xi_\text{P}^{\text{T}},1) U (\xi_\text{P}^{\text{T}},1)^{\text{T}} \geq 0$. Therefore, if $||\xi_\text{P}||_2 \rightarrow \infty$, $s_\text{P} \rightarrow \infty$.  
Therefore, $||\xi_\text{P}||_2$ is bounded by some real number $\Upsilon_P>0$
which depends on $U_i, \rho_i$ and $y$.}
As $\xi \in \mathbb{R}_+^{|\mathcal{K}|}$ and $\rho_i>0$, componentwise, we have
\begin{align*}
    \liminf_{||\xi||_2 \rightarrow \infty}\lambda(\xi)(\rho_i^{\text{T}} \xi - y) \geq 0,
\end{align*}
for any $\lambda(\xi)\ge 0$. Then, 
\begin{align*}
    s_\text{D} &= -1 + (\xi_\text{D}^{\text{T}},1) U_i (\xi_\text{D}^{\text{T}},1)^{\text{T}} + \lambda_\text{D}(\rho_i^{\text{T}} \xi_\text{D} - y),\\
    &=-1 + (\xi_\text{D}^{\text{T}},1) U (\xi_\text{D}^{\text{T}},1)^{\text{T}} + \frac{1}{2^i}||\xi_\text{D}||_2^2 + \frac{1}{2^i} + \lambda_\text{D}(\rho_i^{\text{T}} \xi_\text{D} - y).
\end{align*}
It is clear that $\frac{1}{2^i}||\xi_\text{D}||_2^2 \rightarrow \infty$ and \AL{the} \VS{other terms are lower bounded by some nonnegative number}. \NGU{Therefore, $s_\text{D} \rightarrow \infty$ when $||\xi_\text{D}||_2 \rightarrow \infty$.
\AL{Hence}, $||\xi_\text{D}||_2$ is bounded by some real number $\Upsilon_D >0$ which depends on $U_i, \rho_i$ and $y$.}
Let $\Upsilon = \max(\Upsilon_\text{P},\Upsilon_\text{D})$. Then, \eqref{minmax problem known known} is equivalent to
\begin{align*}
    &s_\text{P} = \min_{\xi \in \mathbb{R}_+^{|\mathcal{K}|},||\xi||_2 \leq \Upsilon} \max_{\lambda \geq 0}\,\,\,\,\,\,\, \mathcal{L}(\lambda,\xi, U^i, \rho^i,y).\nonumber\\
    &s_\text{D} = \max_{\lambda \geq 0}\min_{\xi \in \mathbb{R}_+^{|\mathcal{K}|}, ||\xi||_2 \leq \Upsilon}\,\,\,\,\,\,\, \mathcal{L}(\lambda,\xi, U^i, \rho^i,y).
\end{align*}
Note that the set $\left\{\xi\,\,|\,\,\xi \in \mathbb{R}_+^{|\mathcal{K}|}, ||\xi||_2 \leq \Upsilon\right\}$ is compact. Therefore, from Sion's \NGUNEW{minimax} theorem $s_\text{P} = s_\text{D}$ for every 
$U_i$, $\rho_i$, $i\in \mathbb{N}$. 
\VS{
 For any $\xi$ such that $\rho^{\text{T}} \xi > y$, it is easy to see that
\begin{align*}
    \max_{\lambda \geq 0}\,\,\,\,\,\,\, \mathcal{L}(\lambda,\xi, U, \rho,y)=\infty
\end{align*}
The condition $s_{\text{P}} < \infty$ gives $\rho^{\text{T}} \xi \leq y$ and $\lambda=0$ which in turn implies that
\begin{align*}
    s_{\text{P}} = \min_{\rho^{\text{T}} \xi \leq y} \,\,\,\,\,\,\, \mathcal{L}(0,\xi, U, \rho,y)\ge 0.
\end{align*}
Therefore, 
 \eqref{constraint iii} is equivalent to $s_{\text{D}} \geq 0$. Then, there exists a sequence of nonnegative numbers $\lambda_j \geq 0$ and a decreasing sequence of positive numbers $\theta_j > 0$, such that $\theta_j \rightarrow 0$ as $j\rightarrow \infty$, for which the following condition holds 
\begin{equation}\label{alternative constraint of iii}
\left\{
\begin{aligned}
    &-1 + (\xi^{\text{T}},1) U (\xi^{\text{T}},1)^{\text{T}} + \lambda_j(\rho^{\text{T}} \xi - y) \geq -\theta_j, \ \forall \ \xi \in \mathbb{R}_{+}^{|\mathcal{K}|}, \ j \in \mathbb{N},\\
&    \lambda_j \geq 0, \ \forall \ j \in \mathbb{N}.
\end{aligned}
\right.
\end{equation}
For each $j\in \mathbb{N}$, define
\[
Fea(\theta_j)=\{(U,\rho,y,\lambda) \mid -1 + (\xi^{\text{T}},1) U (\xi^{\text{T}},1)^{\text{T}} + \lambda(\rho^{\text{T}} \xi - y) \ge -\theta_j, \ \lambda \ge 0\}.
\]
The feasible region defined by \eqref{alternative constraint of iii} is equivalent to $\bigcap\limits_{j\in \mathbb{N}} Fea(\theta_j)$.  For any $i<j$, $Fea(\theta_j) \subset Fea(\theta_i)$. Therefore, $Fea(\theta_j) \downarrow \bigcap\limits_{i\in \mathbb{N}} Fea(\theta_i)$ as $j\rightarrow \infty$. The feasible set $Fea(\theta_j)$ as $j\rightarrow \infty$ is given by 
\begin{equation}\label{in-bet-const}
    \left\{
\begin{aligned}
   & (\xi^{\text{T}},1) Z (\xi^{\text{T}},1)^{\text{T}} \geq 0, \ \forall \ \xi \in \mathbb{R}_{+}^{|\mathcal{K}|},\\
   & \lambda \ge 0,
\end{aligned}
\right.
\end{equation}
where $Z \in \mathcal{S}^{|\mathcal{K}|+1}$ and
\begin{align*}
    Z = \begin{pmatrix} -Q & \vline & -\frac{1}{2}q+Q\mu + \lambda \rho\\
    \hline
    -\frac{1}{2}q^{\text{T}}+\mu^{\text{T}}Q + \lambda \rho^{\text{T}}&\vline & -t - \mu^{\text{T}} Q \mu -1  - \lambda y \end{pmatrix}.
\end{align*}
Using similar arguments as above, the constraint \eqref{in-bet-const} is equivalent to 
\begin{gather}\label{final-iii}
Z \in \text{COP}^{|\mathcal{K}|+1}, \ \lambda\ge 0.
\end{gather}
This implies that  the constraint $(\text{iii})$ of \eqref{SUB reformulated SDP unbounded known known nonnegative} is equivalent to \eqref{final-iii}.}
\VS{Hence, DRCCMDP problem \eqref{DRCCMDP-opt} is equivalent to \eqref{matrix form unbounded nonegative polytopic}.}
\item \VS{Let  the distribution of $\hat{R}$ belongs to the uncertainty set $ \mathcal{D}_2(\varphi, \mu, \Sigma,\delta_0)$.
 From Theorem 3.4 \cite{cheng2014distributionally}, the dual of the optimization  problem $\sup_{F \in \mathcal{D} }\mathbb{P}_{\AL{F}}\left(\rho^{\text{T}} \hat{R} \leq y \right)$ can be written as
\begin{align*}
    & \inf\,\,\, (-t-\mu^{\text{T}}q-\mu^{\text{T}}Q\mu+\delta_0 \Sigma \circ Q) \\
    \mbox{s.t.}\,\,\, & (\text{i}) \quad \mathbf{1}_{\left\{\rho^{\text{T}} \xi \leq y\right\}} + t + q^{\text{T}} \xi - \xi^{\text{T}} Q \xi + 2 \mu^{\text{T}}Q\xi \leq 0,\ \forall\ \xi \in \mathbb{R}_+^{|\mathcal{K}|}, \nonumber \\
    & (\text{ii}) \quad Q \in \mathcal{S}_+^{|\mathcal{K}| }, \nonumber
\end{align*}
and the strong duality holds. The rest of the proof follows from the similar arguments used for the case of \AL{the} uncertainty set $\mathcal{D}_1(\varphi,\mu,\Sigma)$.}

\item \VS{If the distribution of $\hat{R}$ belongs to the uncertainty set $\mathcal{D}_3(\varphi,\mu,\Sigma,\delta_1,\delta_2)$,
using  Lemma 1 of \cite{delage2010distributionally} the dual of  the problem $\sup_{F \in \mathcal{D} }\mathbb{P}_{\AL{F}}\left(\rho^{\text{T}} \hat{R} \leq y \right)$ is given by}
\begin{align*}
    & \inf\,\,\, (r+t) \\
    \mbox{s.t.}\,\,\, & (\text{i}) \quad r \geq \mathbf{1}_{\left\{\rho^{\text{T}} \xi \leq y\right\}} - \xi^{\text{T}}Q\xi - \xi^{\text{T}}q,\ \forall\ \xi \in \mathbb{R}_+^{|\mathcal{K}|} , \nonumber \\
    & (\text{ii}) \quad t \geq (\delta_2 \Sigma + \mu \rho^{\text{T}}) \circ Q + \rho^{\text{T}} q + \sqrt{\delta_1} ||\Sigma^{\frac{1}{2}}(q+2Q\mu)||_2, \nonumber\\
    & (\text{iii}) \quad Q \in \mathcal{S}_+^{|\mathcal{K}|}, \nonumber
\end{align*}
\VS{and strong duality holds. Again,  the rest of the proof follows using similar arguments used in the case of $\mathcal{D}_1(\varphi,\mu,\Sigma)$.}
\end{enumerate}
\end{proof}
\begin{remark}
Copositive \AL{optimization} has been studied in the literature. In practical application\AL{s}, the copositive constraints can be approximated conservatively by SDP (semidefinite programming) constraints. We refer to \cite{bomze2017fresh,bomze2019notoriously,xu2018copositive} for some recent researches about SDP approximations.
\end{remark}
\VS{
\section{Statistical distance based uncertainty sets} \label{stat-dist-un}
In this section, we consider uncertainty sets defined using statistical distance metric known as $\phi$-divergence and Wasserstein distance. For each uncertainty set, we propose equivalent reformulation of DRCCMDP problem \eqref{DRCCMDP-opt-ex} (or \eqref{DRCCMDP-opt}).
}
\subsection{\textbf{Uncertainty set with} \texorpdfstring{$\phi$}{Lg} -\textbf{divergence} distance}\label{section phi divergences}
\VS{We consider an uncertainty set defined using 
 statistical distance metric called $\phi$-divergence. 
In such uncertainty set, a nominal distribution is known to the decision maker based on \AL{the} available estimated data. The decision maker \AL{believes} that the true distribution of $\hat{R}$ belongs to a ball of radius $\theta_\phi$ and center\AL{ed} at a nominal distribution $\nu$ and the distance between the true distribution and $\nu$ is given by a $\phi$-divergence. We show that the DRCCMDP problem \eqref{DRCCMDP-opt-ex} is equivalent to an SOCP problem for various $\phi$-divergences.}  
\begin{definition}
The $\phi-$divergence distance between two probability measures $\nu_1$ and $\nu_2$ with densities $f_{\nu_1}$ and $f_{\nu_2}$, respectively, \VS{and full support  $\mathbb{R}^{|\mathcal{K}|}$ is given by}
\begin{align*}
    I_{\phi}(\nu_1,\nu_2) = \int_{\mathbb{R}^{|\mathcal{K}|}}\phi\left(\frac{f_{\nu_1}(\xi)}{f_{\nu_2}(\xi)}\right)f_{\nu_2}(\xi)d \xi.
\end{align*}
\end{definition}
For different choices of $\phi$, we refer to \cite{ben2013robust} and \cite{pardo2018statistical}. Let $\nu \in \NGUNEW{\mathcal{M}^+}$ be a nominal distribution with \AL{a} density function $f_{\nu}$.
 \VS{The uncertainty set  of the distribution of $\hat{R}$ based on $\phi$-divergence is defined by }
\begin{equation}\label{uncertainty phi divergences set}
\mathcal{D}_4(\nu,\theta_\phi) = \left\{F \in \mathcal{M}^+ \mid
I_\phi(F,\nu) \leq \theta_\phi
\right\},
\end{equation}
where $\theta_\phi > 0$.
\begin{definition}
The conjugate of $\phi$ is a function $\phi^* : \mathbb{R} \rightarrow \mathbb{R}\cup\infty$ such that
\begin{align*}
    \phi^*(r) = \sup_{t \geq 0}\left\{rt-\phi(t)\right\}, \ \forall\ r \in \mathbb{R}.
\end{align*}
\end{definition}

\begin{lemma}\label{lemma for general phi-divergences}
Consider an optimization problem \begin{align}\label{definition vP phi divergence}
    \inf_{F\in \mathcal{D}_4(\nu,\theta_\phi)}
    \mathbb{P}_{\AL{F}}(\rho^{\emph{T}} \hat{R} \geq y).
\end{align}
Then, the dual problem of \eqref{definition vP phi divergence} is given by
\begin{align*}
    &\sup_{\lambda > 0, \beta \in \mathbb{R}}\left\{\beta-\lambda \theta_\phi - \lambda \phi^*\left(\frac{-1+\beta}{\lambda}\right)\mathbb{P}_{\nu}(O)- \lambda \phi^*\left(\frac{\beta}{\lambda}\right)(1-\mathbb{P}_{\nu}(O))\right\},\nonumber
\end{align*}
where $O = \left\{\xi \in \mathbb{R}^{|\mathcal{K}|}\,\,\, |\,\,\, \rho^{\emph{T}} \xi \geq y \right\}$, such that the strong duality holds.

\end{lemma}

\begin{proof}
We rewrite the primal problem \eqref{definition vP phi divergence} as a \VS{following semi-infinite programming problem}
\begin{align}\label{formulation of vP}
    v_\text{P} = &\inf_{F \geq 0}\int_{\mathbb{R}^{|\mathcal{K}|}}\mathbf{1}_{O}(\xi)F(\xi)\text{d} \xi \nonumber \\
    \mbox{s.t.}\,\,\,\,& (\text{i}) \quad \int_{\mathbb{R}^{|\mathcal{K}|}}f_{\nu}(\xi)\phi\left(\frac{F(\xi)}{f_{\nu}(\xi)}\right)\text{d} \xi \leq \theta_\phi, \nonumber\\
    & (\text{ii}) \quad \int_{\mathbb{R}^{|\mathcal{K}|}}F(\xi)\text{d} \xi = 1.
\end{align}
The dual problem of \eqref{formulation of vP} is given by
\begin{align*}
    &v_\text{D} = \\
    &\sup_{\lambda \geq 0, \beta \in \mathbb{R}} \left\{\beta-\lambda \theta_\phi + \inf_{F(\xi)\geq 0}\left\{\int_{\mathbb{R}^{|\mathcal{K}|}}\left(\mathbf{1}_{O}(\xi)F(\xi) - \beta F(\xi) + \lambda f_{\nu}(\xi)\phi\left(\frac{F(\xi)}{f_{\nu}(\xi)}\right)\right)\text{d} \xi\right\}\right\},
\end{align*}
where $\lambda$ is the dual variable of the constraint $(\text{i})$ of \eqref{formulation of vP} and $\beta$ is the dual variable of the constraint $(\text{ii})$ of \eqref{formulation of vP}. Since $\theta_{\phi} > 0$, the Slater's condition holds \VS{which implies that }the strong duality holds, i.e., $v_\text{P}=v_\text{D}$. \VS{The rest of the proof follows from  Theorem 1 of \cite{jiang2016data}.} 
\end{proof}

\NGU{We study 4 cases of $\phi-$divergences whose conjugates are given in Table \ref{table of phi divergences}}.
\begin{table}[ht]\small 
\centering
\caption{List of selected $\phi-$divergences with their conjugate} \vspace{.2cm}
\begin{tabular}{|*{5}{c|}}
\hline
Divergence & $\phi(t), t \geq 0$ & $\phi^*(r)$\\
\hline
Kullback-Leibler & $t \log(t)-t+1$. & $\text{e}^r-1$\\   
\hline
Variation distance & $|t-1|$. &$\makecell{-1, \quad \quad r \leq -1,\\r, \quad \quad -1 \leq r \leq 1,\\\infty, \quad \quad r>1.}$ \\
\hline
Modified $\chi^2$ - distance & $(t-1)^2$. & $\makecell{-1,\quad \quad r\leq -2,\\r+\frac{r^2}{4},\quad \quad r>-2.}$\\
\hline
Hellinger distance & $(\sqrt{t}-1)^2$. & $\makecell{\frac{r}{1-r},\quad \quad r<1,\\\infty,\quad \quad r\geq 1.}$\\
\hline
\end{tabular}\label{table of phi divergences}
\end{table}
Using Lemma \ref{lemma for general phi-divergences}, \VS{the following result holds.}
\begin{theorem} \label{theorem phi divergencess}
\VS{Consider the DRCCMDP problem \eqref{DRCCMDP-opt-ex} under the uncertainty set defined by \eqref{uncertainty phi divergences set} for the $\phi$-divergences listed in Table \ref{table of function f}. If the reference distribution $\nu$ is a normal distribution with mean vector $\mu_{\nu}$ and positive definite covariance matrix $\Sigma_{\nu}$, the DRCCMDP problem \eqref{DRCCMDP-opt-ex} is equivalent to the following \emph{SOCP} problem}
\begin{align}\label{main reformulation SOCP phi divergences}
    &\max \quad y \nonumber\\
     \mbox{s.t.} \,\,\,\,\,& (\emph{i}) \quad \rho^{\emph{T}}\mu_{\nu}-\Phi^{(-1)}[f(\theta_\phi,\epsilon)]\|\Sigma_\nu^{\frac{1}{2}}\rho\|_2\geq y,\nonumber\\& (\emph{ii}) \quad \rho\in\mathcal{Q}_{\alpha}(\gamma),
\end{align}
\VS{where $\Phi^{(-1)}$ is the quantile of the standard normal distribution and the values of $\theta_\phi$, $\epsilon$ and $f(\theta_\phi,\epsilon)$ 
 for different $\phi$-divergences are given in Table \ref{table of function f}.
}
\begin{table}[ht]\small 
\centering
\caption{The function $f$ for selected $\phi-$divergences} \vspace{.2cm}
\scalebox{0.84}{\begin{tabular}{|*{5}{c|}}
\hline
Divergence & $f(\theta_\phi,\epsilon)$ & $\theta_\phi$, $\epsilon$\\
\hline
Kullback-Leibler & $\inf_{x \in (0,1)} \frac{\emph{e}^{-\theta_\phi}x^{1-\epsilon}-1}{x-1}$ & $\theta_\phi>0$, $0<\epsilon<1$ \\   
\hline
Variation distance & $1-\epsilon+\frac{\theta_\phi}{2}$ & $\theta_\phi>0$, $0<\epsilon<1$\\
\hline
Modified $\chi^2$ - distance & $\makecell{1-\epsilon + \frac{\sqrt{\theta_\phi^2+4\theta_\phi(\epsilon-\epsilon^2)}-(1-2\epsilon)\theta_\phi}{2 \theta_\phi+2}}$ & $\theta_\phi>0$, $0 < \epsilon < \frac{1}{2}$\\
\hline
Hellinger distance & $\makecell{\frac{-B+\sqrt{\Delta}}{2},\\\mbox{where}\quad B=-(2-(2-\theta_\phi)^2)\epsilon-\frac{(2-\theta_\phi)^2}{2},\\C=\left(\frac{(2-\theta_\phi)^2}{4}-\epsilon\right)^2,\\ \Delta=B^2-4C=(2-\theta_\phi)^2\left[4-(2-\theta_\phi)^2\right]\epsilon(1-\epsilon).}$ & $0 < \theta_\phi < 2 - \sqrt{2}$, $0<\epsilon<1$\\
\hline
\end{tabular}\label{table of function f}}
\end{table}
\end{theorem}
\begin{proof}
Using Lemma \ref{lemma for general phi-divergences}, we prove that the constraint $(\text{i})$ of \eqref{DRCCMDP-opt-ex} is equivalent to the following constraint
\begin{align}\label{equi form}
    \mathbb{P}_{\nu}(\rho^{\text{T}}\hat{R} \geq y) \geq f(\theta_\phi,\epsilon).
\end{align}
Since $\nu$ is a normal distribution with mean vector $\mu_\nu$ and covariance matrix $\Sigma_\nu$, it is well known that \eqref{equi form} is equivalent to the constraint $(\text{i})$ of \eqref{main reformulation SOCP phi divergences}. The details of the proof for the Hellinger distance case is given in Appendix \ref{app_Hellinger}. The proofs for Kullback-Leibler, Variation distance and Modified $\chi^2$ - distance follow from Propositions 2, 3 and 4 of \cite{jiang2016data}.
\end{proof}

\subsection{Uncertainty set with Wasserstein distance}\label{Wasser-metric}
 We consider an uncertainty set defined using statistical distance metric called Wasserstein distance.
We show that the DRCCMDP problem \eqref{DRCCMDP-opt} is tractable if the reference distribution $\nu$ follows a discrete distribution whose scenarios are taken from historical data. We refer to Villani \cite{villani2009optimal,villani2021topics} for more details of the Wasserstein distance metric.

Let $\varphi$ be a closed, convex subset of $\mathbb{R}^{|\mathcal{K}| }$ and $p \in [1,\infty)$. Let $\mathcal{B}(\varphi)$ denote\AL{s} the Borel $\sigma-$ algebra on $\varphi$. Let $\mathcal{P}(\varphi)$ be the set of all probability measures defined on $\mathcal{B}(\varphi)$ and $\mathcal{P}_p(\varphi)$ denote the subset of $\mathcal{P}(\varphi)$ with finite $p-$ moment \VS{and it is defined as}
\begin{align*}
    \mathcal{P}_p(\varphi) = \left\{\mu \in \mathcal{P}(\varphi)\mid\int_{\xi \in \varphi}||\xi-\xi_0||_2^p \mu(\text{d} \xi)<\infty \,\,\,\mbox{for some} \,\,\, \xi_0 \in \varphi \right\}.
\end{align*}

It follows from the triangle inequality that the above definition of $\mathcal{P}_p(\varphi)$ does not depend on $\xi_0$.

\begin{definition}[Wasserstein distance]
The Wasserstein distance $W_p(\mu,\nu)$ between 

\noindent$\nu_1,\nu_2 \in \mathcal{P}_p(\varphi)$ is defined by
\begin{align*}
    W_p(\nu_1,\nu_2) = \left(\inf_{\gamma \in \mathcal{P}_{\nu_1,\nu_2}(\varphi \times \varphi)}\int_{\varphi \times \varphi}||x-z||_2^p \gamma(dx,dz)\right)^{\frac{1}{p}},
\end{align*}
where $\mathcal{P}_{\nu_1,\nu_2}(\varphi \times \varphi)$ denotes the set of all probability measures defined on $\mathcal{B}(\varphi \times \varphi)$ such that the marginal laws are $\nu_1$ and $\nu_2$.

\end{definition}
\VS{The uncertainty set using Wasserstein distance is defined by}
\begin{equation}\label{uncertainty Wasserstein set}
\mathcal{D}_5(\varphi,\nu,p,\theta_W) = \left\{F \in \mathcal{P}_p(\varphi)\mid
W_p(F,\nu) \leq \theta_W
\right\},
\end{equation}
where $\nu \in \mathcal{P}_p(\varphi)$ and $\theta_W > 0$.
\begin{lemma}\label{main lemma Wasserstein}
Consider an optimization problem 
\begin{align}\label{definition vP Wasserstein}
    \sup_{F\in \mathcal{D}_5(\varphi,\nu,p,\theta_W)}
    \mathbb{P}_{\AL{F}}(\rho^{\emph{T}} \hat{R} \leq y).
\end{align}
Then, the dual problem of \eqref{definition vP Wasserstein} is given by
\begin{align}\label{reformulation Wasserstein final}
   &\inf_{\lambda \geq 0} \left\{\lambda \theta_W^p -\int_{\varphi} \inf_{z \in \varphi}\left[\lambda||x-z||_2^p-\mathbf{1}_{\left\{\rho^{\text{T}} z \leq y\right\}}\right]\nu(\emph{d}x)\right\},
\end{align}
such that the strong duality holds \VS{and} \AL{the} optimal value\AL{s} of \AL{the} primal and \AL{the} dual problems are finite.

\end{lemma}

\begin{proof}
\VS{
Let $\Xi$ be a  Polish space with metric $d$, $\mathcal{P}(\Xi)$ be the set of Borel probability measures on $\Xi$, $\nu \in \mathcal{P}(\Xi)$ and $\Psi \in L^1(\nu)$, where $L^1(\nu)$ represents the $L^1$ space of $\nu$ - measurable functions. 
It follows from Theorem 1 of \cite{gao2016distributionally} that
the following strong duality holds}
\begin{align} \label{strong duality wasserstein}
    &\sup_{\mu \in \mathcal{P}(\Xi)}\left\{\int_{\Xi}\Psi(\xi)\mu(d\xi) \,\mid \,W_p(\mu,\nu)\leq \theta_W \right\}\nonumber\\
   & = \inf_{\lambda \in \mathbb{R}, \lambda \geq 0}\left\{\lambda \theta_W^p-\int_{\Xi}\inf_{\xi \in \Xi}\left[\lambda d^p(\xi,\zeta)-\Psi(\xi)\right]\nu(\text{d}\zeta)\right\} <\infty,
\end{align}
\VS{provided the growth factor given by Definition 4 of \cite{gao2016distributionally} is finite.}
We apply this result in our case by choosing $\Xi = \varphi$,  $d$  as \AL{an} Euclidean metric and 
$\Psi(\xi) = \mathbf{1}_{\left\{\rho^{\text{T}} \xi \leq y\right\}}$ for all $\xi \in \varphi$. \VS{For this choice of $\Psi(\xi)$, it is easy to see from Definition 4 of \cite{gao2016distributionally} that the growth factor is zero. } 
Since $\left\{\xi \in \varphi\,\,|\,\,\rho^{\text{T}} \xi \leq y\right\}$ is a closed set, it is a Borel measurable set. Hence, it is clear that $\Psi \in L^1(\nu)$ for all $\nu \in \mathcal{P}(\varphi)$.  Then,  \eqref{strong duality wasserstein} reduces to
\begin{align*}
    \sup_{F \in \mathcal{D}_5(\varphi,\nu,p,\theta_W) }\mathbb{P}_{\AL{F}}\left(\rho^{\text{T}} \hat{R} \leq y \right) = \inf_{\lambda \geq 0}\left\{\lambda \theta_W^p-\int_{\varphi}\inf_{\xi \in \varphi}\left[\lambda ||\zeta-\xi||^p_2-\mathbf{1}_{\left\{\rho^{\text{T}} \xi \leq y\right\}}\right]\nu(\text{d}\zeta)\right\}.
\end{align*}

\end{proof}
\VS{We consider the case when $p=1$ and $\nu$ is a data-driven reference distribution, i.e., it is a discrete distribution with $H$ scenarios $\Tilde{\xi}_1,\ldots,\Tilde{\xi}_H$, where $\Tilde{\xi}_i \in \varphi$, for every $i=1,\ldots,H$. Using Lemma \ref{main lemma Wasserstein}, we propose a deterministic reformulation of the DRCCMDP problem \eqref{DRCCMDP-opt}.}

\begin{lemma}\label{reformulation wassertein ok}
If the distribution of $\hat{R}$ belongs to the uncertainty set defined by \eqref{uncertainty Wasserstein set}, the \emph{DRCCMDP} \eqref{DRCCMDP-opt} can be reformulated equivalently as the following deterministic problem
\begin{align}\label{reformulation Wasserstein p=1 data-driven}
    &\sup\,\,\,\,y \nonumber\\
    \mbox{s.t.} \,\,\,\,\,
    & (\emph{i}) \quad \theta_W - \frac{1}{H}\sum_{i=1}^H g_i \leq l \epsilon,\nonumber\\
    &  (\emph{ii}) \quad \inf_{z \in \varphi, \rho^{\emph{T}}z \leq y}||\Tilde{\xi_i}-z||_2\geq l + g_i, \ \forall \ i = 1,\ldots,H,\nonumber\\
    & (\emph{iii}) \quad  l > 0, \ \rho\in\mathcal{Q}_{\alpha}(\gamma), \ g_i \leq 0, \ \forall \ i = 1,\ldots,H. 
\end{align}
\end{lemma}

\begin{proof}
Using Lemma \ref{main lemma Wasserstein}, since $\nu$ is a discrete distribution with $H$ scenarios $\Tilde{\xi_1},...,\Tilde{\xi_H}$, the constraint $(\text{i})$ of \eqref{DRCCMDP-opt} \VS{can be equivalently written as}
\begin{align*}
  \lambda \theta_W - \frac{1}{H}\sum_{i=1}^H \inf_{z \in \varphi}\left[\lambda||\Tilde{\xi_i}-z||_2-\mathbf{1}_{\left\{\rho^{\text{T}}z\leq y\right\}}\right] \leq \epsilon, \ 
 \VS{\lambda \ge 0}.
\end{align*}
By introducing auxiliary variables $t_i$, $i = 1,...,H$, the above constraint can be rewritten as 
\begin{equation}\label{inequality 1}
\begin{cases}
    & (\text{i}) \quad \lambda \theta_W - \frac{1}{H}\sum_{i=1}^H t_i \leq \epsilon, \ \lambda \ge 0\\
    & (\text{ii}) \quad \inf_{z \in \varphi}\left[\lambda||\Tilde{\xi_i}-z||_2-\mathbf{1}_{\left\{\rho^{\text{T}}z \leq y\right\}}\right] \geq t_i, \ \forall\ i = 1,\ldots,H.
\end{cases}
\end{equation}
The constraint $(\text{ii})$ of \eqref{inequality 1} is equivalent to the following two constraints
\begin{equation}\label{fc1}
\begin{cases}
    & (\text{i}) \quad \inf_{z \in \varphi}\lambda||\Tilde{\xi_i}-z||_2\geq t_i, \ \forall \ i = 1,\ldots,H,\\
    & (\text{ii}) \quad \inf_{z \in \varphi,\rho^{\text{T}}z \leq y}\lambda||\Tilde{\xi_i}-z||_2-1 \geq t_i, \ \forall \ i = 1,\ldots,H.
\end{cases}
\end{equation}
Since $\lambda \geq 0$, $\inf_{z \in \varphi}\lambda||\Tilde{\xi_i}-z||_2 = 0$. Then, the constraint $(\text{i})$ of \eqref{fc1} is equivalent to $t_i \leq 0$, for every $i = 1,\ldots,H$. Moreover, if $\lambda=0$,  from the constraint $(\text{ii})$ of \eqref{fc1},  $t_i \leq -1$, for every $i = 1,\ldots,H$, which in turn implies $-\frac{1}{H}\sum_{i=1}^H t_i \geq 1$. This 
violates the constraint $(\text{i})$ of \eqref{inequality 1}. Hence, $\lambda>0$. Let $l= \frac{1}{\lambda}$ and $g_i=\frac{t_i}{\lambda}$, for every $i = 1,\ldots,H$. Therefore, the constraint $(\text{i})$ of  \eqref{DRCCMDP-opt} is equivalent to the following constraints
\begin{equation}
\begin{cases}
    & (\text{i}) \quad \theta_W - \frac{1}{H}\sum_{i=1}^H g_i \leq l\epsilon,\\
    & (\text{ii}) \quad \inf_{z \in \varphi, \rho^{\text{T}}z \leq y}||\Tilde{\xi_i}-z||_2\geq l+ g_i, \ \forall \ i = 1,\ldots,H,\\
    & (\text{iii}) \quad l > 0, \ \VS{g_i \le 0, \ \forall \ i = 1,\ldots,H.}
\end{cases}
\end{equation}
\VS{This implies that the DRCCMDP \eqref{DRCCMDP-opt} is equivalent to \eqref{reformulation Wasserstein p=1 data-driven}.}
\end{proof}
\VS{The constraint $(\text{ii})$ of \eqref{reformulation Wasserstein p=1 data-driven} includes $\inf$ term which makes it difficult to solve the problem directly. 
We show that the  optimization problem \eqref{reformulation Wasserstein p=1 data-driven} is equivalent to a MISOCP problem and a biconvex optimization problem for the case of full support and nonnegative support, respectively.
}

\subsubsection{DRCCMDP under Wasserstein distance based uncertainty set with full support}
\begin{lemma}\label{transformation lemma full support}
If $\varphi=\mathbb{R}^{|\mathcal{K}|}$,
\begin{align*} 
    \inf_{\rho^{\emph{T}}z \leq y}||\Tilde{\xi_i}-z||_2 = \max\left(0,\frac{\rho^{\emph{T}} \Tilde{\xi_i}-y}{||\rho||_2}\right), \ \forall \ i = 1,\ldots,H.
\end{align*}
\end{lemma}

The proof is given in Appendix \ref{app_C}.
Using Lemma \ref{transformation lemma full support}, we have the following result.

\begin{lemma}\label{lemma intermediare}
The optimization problem \eqref{reformulation Wasserstein p=1 data-driven} is equivalent to the following optimization problem 
\begin{align}\label{deterministic wasserstein full support}
    &\sup\,\,\,\,y \nonumber\\
    \mbox{s.t.} \,\,\,\,\,
    & (\emph{i}) \quad \beta \theta_W - \frac{1}{H}\sum_{i=1}^H b_i \leq t \epsilon,\nonumber\\
    & (\emph{ii}) \quad \max\left(0,\rho^{\emph{T}}\Tilde{\xi_i}-y \right) \geq b_i+t, \ \forall \ i = 1,\ldots,H,\nonumber\\
    & (\emph{iii}) \quad ||\rho||_2 \leq \beta, \ t \geq 0, \ \beta > 0,  \rho\in\mathcal{Q}_{\alpha}(\gamma), \ b_i \leq 0,\ \forall \ i = 1,\ldots,H.
\end{align}
\end{lemma}

\begin{proof}\VS{
Using Lemma \ref{transformation lemma full support}, the constraint $(\text{ii})$ of problem \eqref{reformulation Wasserstein p=1 data-driven} can \AL{be} written as 
\[
\max\left(0,\frac{\rho^{\text{T}} \Tilde{\xi_i}-y}{||\rho||_2}\right)\geq l + g_i, \ \forall \ i = 1,...,H.
\]
Let $\beta > 0$ be an auxiliary variable. Then, under the transformations  $t=\beta l$, $b_i=\beta g_i$, for every $i = 1,...,H$, it is easy to see that \eqref{reformulation Wasserstein p=1 data-driven} is equivalent to
\eqref{deterministic wasserstein full support}.}
\end{proof}
\VS{It is clear that a vector $(y,\rho,\beta,(b_i)_{i=1}^H,t)$ such that $\rho\in \mathcal{Q}_{\alpha}(\gamma)$,   $\beta = ||\rho||_2, b_i=0$, for every  $i = 1,\ldots,H$,  $t=\frac{\theta_W}{\epsilon}||\rho||_2$ and $y=\min_{i = 1,\ldots,H}(\rho^{\text{T}}\Tilde{\xi_i})-\frac{\theta_W}{\epsilon}||\rho||_2$ is a feasible solution of \eqref{deterministic wasserstein full support}.
Therefore, the optimal solutions of \eqref{deterministic wasserstein full support} and the following optimization problem are \AL{the} same}
\begin{align}\label{extra deterministic wasserstein full support}
    &\sup\,\,\,\,y \nonumber\\
    \mbox{s.t.} \,\,\,\,\,
    & (\text{i}) \quad \beta \theta_W - \frac{1}{H}\sum_{i=1}^H b_i \leq t \epsilon,\nonumber\\
    & (\text{ii}) \quad \max\left(0,\rho^{\text{T}}\Tilde{\xi_i}-y \right) \geq b_i+t, \ \forall \ i = 1,\ldots,H,\nonumber\\
    & (\text{iii}) \quad y \geq \min_{i = 1,\ldots,H}(\rho^{\text{T}}\Tilde{\xi_i})-\frac{\theta_W}{\epsilon}||\rho||_2,\nonumber\\
    & (\text{iv}) \quad ||\rho||_2 \leq \beta, \ t \geq 0, \ \beta > 0,  \rho\in\mathcal{Q}_{\alpha}(\gamma), \ b_i \leq 0,\ \forall \ i = 1,\ldots,H.
\end{align}
\VS{We reformulate the problem \eqref{extra deterministic wasserstein full support} as a\AL{n} MISOCP problem. In order to do that, we  define a constant $M=\left(\frac{\theta_W}{\epsilon}+2 \max_{i=1,\ldots,H}||\Tilde{\xi_i}||_2\right)$ for which the following result holds.}
\begin{lemma}\label{sufficient condition to apply big-M method}
\VS{ For every feasible solution of \eqref{extra deterministic wasserstein full support}, $M \geq |y-\rho^{\emph{T}} \Tilde{\xi_i}|$ for all $i = 1,\ldots,H$.}
\end{lemma}

The proof is given in Appendix \ref{app_D}.

\begin{theorem}
Consider the \emph{DRCCMDP}  problem \eqref{DRCCMDP-opt}. We assume that the distribution of $\hat{R}$ belongs to the uncertainty set defined by \eqref{uncertainty Wasserstein set} and $\varphi=\mathbb{R}^{|\mathcal{K}|}$. Then, the \emph{DRCCMDP} \eqref{DRCCMDP-opt} can be reformulated equivalently as the following MISOCP
\begin{align}\label{reformulation Wasserstein p=1 data-driven big M-method}
    &\max \quad y \nonumber\\
    \mbox{s.t.} \,\,\,\,\,
    & (\emph{i}) \quad \beta \theta_W - \frac{1}{H}\sum_{i=1}^H b_i \leq t \epsilon,\nonumber\\
    & (\emph{ii}) \quad M \eta_i \geq b_i+t, \ \forall \ i = 1,\ldots,H,\nonumber\\
    & (\emph{iii}) \quad M(1-\eta_i) + \rho^{\emph{T}} \Tilde{\xi_i}-y \geq b_i+t, \ \forall \ i = 1,\ldots,H,\nonumber\\
    & (\emph{iv}) \quad \eta_i \in \left\{0,1\right\}, \ \forall \ i = 1,\ldots,H, \nonumber\\
    & (\emph{v})\quad ||\rho||_2 \leq \beta, t\geq0, \beta >0, \rho\in\mathcal{Q}_{\alpha}(\gamma), b_i \leq 0, \ \forall \ i = 1,\ldots,H.
\end{align}
\end{theorem}
\AL{Notice that the parameter M is the well known big-M constant.}
\begin{proof}
\VS{Since, the distribution of $\hat{R}$ belongs to the uncertainty set defined by \eqref{uncertainty Wasserstein set}, the DRCCMDP problem is equivalent to \eqref{extra deterministic wasserstein full support}. We show that \eqref{extra deterministic wasserstein full support} and \eqref{reformulation Wasserstein p=1 data-driven big M-method} are equivalent. }
\NGU{It is clear that a vector $(y,\rho,\beta,(b_i)_{i=1}^H,(\eta_i)_{i=1}^H,t)$ such that $\rho\in \mathcal{Q}_{\alpha}(\gamma)$,   $\beta = ||\rho||_2, b_i=0$, $t=\frac{\theta_W}{\epsilon}||\rho||_2$, $\eta_i=1$, for every $i = 1,\ldots,H$, and $y=\min_{i = 1,\ldots,H}(\rho^{\text{T}}\Tilde{\xi_i})-\frac{\theta_W}{\epsilon}||\rho||_2$ is a feasible solution of \eqref{reformulation Wasserstein p=1 data-driven big M-method}. 
\VS{Therefore, the optimal solution of \eqref{reformulation Wasserstein p=1 data-driven big M-method} does not change if we add constraint \eqref{inbet-E1} given below 
\begin{equation}\label{inbet-E1}
y \geq \min_{i = 1,\ldots,H}(\rho^{\text{T}}\Tilde{\xi_i})-\frac{\theta_W}{\epsilon}||\rho||_2,
\end{equation}
to the feasible region of \eqref{reformulation Wasserstein p=1 data-driven big M-method}. 
Now, it is enough to show that the constraint $(\text{ii})$ of \eqref{extra deterministic wasserstein full support} is equivalent to $(\text{ii})-(\text{iv})$ of \eqref{reformulation Wasserstein p=1 data-driven big M-method}. 
}
Let the constraint $(\text{ii})$ of  
\eqref{extra deterministic wasserstein full support} be satisfied, i.e.,
\begin{align}\label{nm1}
\max\left(0,\rho^{\text{T}}\Tilde{\xi_i}-y \right) \geq b_i+t,\ \forall \ i =1,\ldots,H.
\end{align}
For each $i=1,\ldots,H$, we consider two cases as follows:}

\noindent\textbf{Case 1:} If $\max\left(0,\rho^{\text{T}}\Tilde{\xi_i}-y \right)=0$, by choosing $\eta_i=0$, \eqref{nm1} is equivalent to the constraint $(\text{ii})$ of \eqref{reformulation Wasserstein p=1 data-driven big M-method}. Moreover, using Lemma \ref{sufficient condition to apply big-M method}, we have
\[
M \geq |y-\rho^{\text{T}}\Tilde{\xi_i}|.
\]
Therefore,
\[
M(1-\eta_i) + \rho^{\text{T}}\Tilde{\xi_i}-y\ge M-|y-\rho^{\text{T}}\Tilde{\xi_i}|\ge 0\ge b_i+t.
\]
\noindent\textbf{Case 2:} If $\max\left(0,\rho^{\text{T}}\Tilde{\xi_i}-y \right)=\rho^{\text{T}}\Tilde{\xi_i}-y$, by choosing $\eta_i=1$, \eqref{nm1} is equivalent to the constraint $(\text{iii})$ of \eqref{reformulation Wasserstein p=1 data-driven big M-method}. Moreover, using Lemma \ref{sufficient condition to apply big-M method}, we have
\[
 M \eta_i = M \geq \rho^{\text{T}}\Tilde{\xi_i}-y \geq b_i+t.
 \]
This implies that there exists $\eta_i\in \{0,1\}$ such that $(\text{ii})-(\text{iv})$ of \eqref{reformulation Wasserstein p=1 data-driven big M-method} are satisifed. Conversely, suppose $(\text{ii})-(\text{iv})$ of  \eqref{reformulation Wasserstein p=1 data-driven big M-method} has a feasible solution. If $\eta_i=1$,  the constraint $(\text{iii})$ of \eqref{reformulation Wasserstein p=1 data-driven big M-method} implies the constraint $(\text{ii})$ of \eqref{extra deterministic wasserstein full support}. If $\eta_i=0$,  the constraint $(\text{ii})$ of \eqref{reformulation Wasserstein p=1 data-driven big M-method} implies the constraint $(\text{ii})$ of \eqref{extra deterministic wasserstein full support}. 
\end{proof}

\begin{remark}
\VS{An MISOCP problem can be solved efficiently \AL{with} BONMIN, PAJARITO or BARON solvers.}
\end{remark}

\subsubsection{DRCCMDP under Wasserstein distance based uncertainty set with nonnegative support}
\begin{lemma}\label{strong duality nonnegative wasserstein}
Let $\varphi=\mathbb{R}_+^{|\mathcal{K}|}$ and consider an optimization problem
\begin{align}\label{dual non wassertein}
    \inf_{z \in \varphi, \rho^{\emph{T}}z \leq y}||\Tilde{\xi_i}-z||_2.
\end{align}
 The dual problem of \eqref{dual non wassertein} is given by
\begin{align*}
    &\max\quad \lambda_i(\rho^{\emph{T}}\Tilde{\xi_i}-y) - \zeta_i^{\emph{T}}\Tilde{\xi_i} \nonumber\\
    \mbox{s.t.} \,\,\,\,\,
    &  \quad ||\zeta_i-\lambda_i \rho||_2 \leq 1, \ \zeta_i \in \mathbb{R}_+^{|\mathcal{K}|}, \lambda_i \geq 0,\ 
\end{align*}
such that the strong duality holds.
\end{lemma}

The proof is given in Appendix \ref{app_E}.

\begin{theorem}\label{theorem nonnegative wasserstein}
Consider the \emph{DRCCMDP}  problem \eqref{DRCCMDP-opt}. We assume that the distribution of $\hat{R}$ belongs to the uncertainty set defined by \eqref{uncertainty Wasserstein set} and $\varphi=\mathbb{R}^{|\mathcal{K}|}_+$. Then, the \emph{DRCCMDP} \eqref{DRCCMDP-opt} can be reformulated equivalently as the following biconvex optimization problem
\begin{align}\label{nonnegative wasserstein continuous optimization}
    &\max \quad y \nonumber\\
    \mbox{s.t.} \,\,\,\,\,
    & (\emph{i}) \quad \theta_W - \frac{1}{H}\sum_{i=1}^H g_i \leq l \epsilon,\nonumber\\
    & (\emph{ii}) \quad \lambda_i(\rho^{\emph{T}}\Tilde{\xi_i}-y) - \zeta_i^{\emph{T}}\Tilde{\xi_i} \geq l + g_i, \ \forall \ i = 1,\ldots,H,\nonumber\\
    & (\emph{iii}) \quad ||\zeta_i-\lambda_i \rho||_2 \leq 1, \ \forall \ i = 1,\ldots,H,\nonumber\\
    & (\emph{iv}) \quad \lambda_i \geq 0, \zeta_i \in \mathbb{R}_+^{|\mathcal{K}|}, \ l > 0, \ g_i \leq 0, \ \rho\in\mathcal{Q}_{\alpha}(\gamma), \ \forall \ i = 1,\ldots,H.
\end{align}
\end{theorem}
The proof follows directly from Lemma \ref{reformulation wassertein ok} and Lemma \ref{strong duality nonnegative wasserstein}.
\begin{remark}
The optimization problem \eqref{nonnegative wasserstein continuous optimization} is a non-convex reformulation with biconvex terms. It can be solved by DMCP solver in CVXPY or nonlinear nonconvex optimization solvers, e.g., IPOPT without any guarantee of running time.
\end{remark}

\section{Machine replacement problem} \label{machine_rep}
\VS{In this section, we consider a machine replacement problem  where a machine in a factory has a life-time of $N$ years. At every stage a maintenance of the machine is scheduled but a factory owner can decide  whether to repair or  do not repair the machine. There is a high probability that the machine behaves like a new one if it is being repaired and its life gets reduced by a year if it is not being repaired.
The factory owner incurs maintenance cost if he decides to repair the machine. It can be modelled as an MDP problem where the life of a machine represents the state of underlying Markov chain, i.e., there are $N+1$ states.
The first state represents a brand new machine. At each state there are two actions: i) "repair", ii) "do not repair". The maintenance cost corresponding to every state-action pair is not exactly known and is realised after the decision is made. Therefore, it is modelled with a random variable.  We assume that for every state action pair $(s,a)$, the maintenance cost is defined as $\hat{c}(s,a)= K + \hat{Z}(s,a)$, where $K$ represents the fixed cost and $\hat{Z}(s,a)$ represents a variable cost which is  a random variable. The machine generates \AL{a} revenue  
$L(s,a)$ at state-action pair $(s,a)$ and the profit for each  $(s,a)\in \mathcal{K}$ is given by 
\begin{align}\label{running reward formula}
\hat{R}(s,a) = L(s,a)-K-\hat{Z}(s,a).    
\end{align}
The factory owner is interested in maximizing the expected discounted profit. We assume that the factory owner has a finite number of the same machines which are modelled using the same Markov chain. Therefore, we compute the optimal repair policy  with respect to a single machine and the same repair policy can be applied for all other machines.}

All the numerical results below are performed using Python 3.8.8 on an Intel Core i5-1135G7, Processor 2.4 GHz (8M Cache, up to 4.2 GHz), RAM 16G, 512G SSD. \VS{We compare the performance of DRCCMDP for each uncertainty set with 
the CCMDP model  \eqref{CCMDP-equi} where the distribution of $\hat{R}$ is assumed to be a normal distribution. 
 In our numerical experiments, we \AL{set the} number of states to 10, the threshold value $\epsilon = 0.1$, the discount parameter $\alpha=0.85$ and the initial distribution of states $\gamma$ to be uniformly distributed. For the above instance, $|\mathcal{K}|=20$ and $\hat{R}$ is a $20\times 1$ random vector with mean vector $\mu$ given by
 \begin{align}\label{mean}
 \mu(s,a)=L(s,a)-K-\mu_{\hat{Z}}(s,a),
 \end{align}
 where $\mu_{\hat{Z}}$ is the mean vector of the random cost vector  $\hat{Z}$.} We take $K=10$, \AL{the functio} \NGU{L} and the mean cost $\mu_{\hat{Z}}$ corresponding to each state-action pair are summarized in Table \ref{1}. 
 For example, at state $1$, if the "repair" action is taken, the factory owner has to pay a random cost with mean $\mu_{\hat{Z}}(1,1) = 10$. If the action  "do not repair" is taken, the mean value of the random cost is $\mu_{\hat{Z}}(1,2) = 0$. \VS{The last state is considered to be risky and not repairing may lead to the machine breakdown. This is the reason we take \AL{the} mean cost equal to 5 if "do not repair" action is taken at state 10.}
 The covariance matrix $\Sigma$ of $\hat{R}$ is randomly generated using the following formula
 \begin{align}\label{cov}
 \Sigma=\frac{AA^{\text{T}}}{20} + D_{20},
 \end{align}
 where $A$ is a $20 \times 20$ random matrix whose all the entries are real numbers belonging to $[0,1]$ generated by the command "\textit{A=numpy.random.random(size=(20, 20))}", $D_{20}$ is a $20 \times 20$ diagonal  matrix with $D_{20}(10,10) = 4$, $D_{20}(20,20) = 9$, $D_{20}(i,i)=1$, for every $i \neq 10,20$ and all other entries equal to zero. \VS{For the above instance, $\Sigma$ is diagonally dominant with high values at entries $(10,10)$ and $(20,20)$ which is due to the fact that action at risky state can have large variance corresponding to both actions. We use \AL{the} above $\mu$ and $\Sigma$ for all the moments based uncertainty sets.} 
 \begin{table}[ht]\small 
\begin{minipage}{.6\linewidth}
\centering
\caption{Random cost $\hat{Z}$ and Revenue $L$}\label{1}
\vspace{.2cm}
\scalebox{0.8}{\begin{tabular}{|*{5}{c|}}
\hline
\backslashbox{State(s) \kern-3em}{\kern-1em Action(a)} & \makecell{"Repair"\\$\mu_{\hat{Z}}(s,1)$} & \makecell{"Do not\\repair"\\$\mu_{\hat{Z}}(s,2)$} & \makecell{"Repair"\\$L(s,1)$} & \makecell{"Do not\\repair"\\$L(s,2)$} \\
\hline
1 & 10 &0 & 30 & 30\\ 
\hline
2 & 10.1 &0 & 30 & 29.9\\ 
\hline
3 & 10.2 &0 & 30 & 29.8\\ 
\hline
4 &10.3 &0 & 30 & 29.7\\ 
\hline
5 &10.4 &0 & 30 & 29.6\\  
\hline
6 &10.5 &0 & 30 & 29.5\\ 
\hline
7 &10.6 &0 & 30 & 29.4\\ 
\hline
8 &10.7 &0 & 30 & 29.3\\ 
\hline
9 & 10.8 &0 & 30 & 29.2\\ \hline
10 & 10.9 &5 & 30 & 29.1\\
\hline
\end{tabular}}
\end{minipage}
\hfill
\begin{minipage}{.35\linewidth}
\caption{Other parameters}\label{2}
\centering
\vspace{.2cm}
\scalebox{0.791}
{\begin{tabular}{|*{4}{c|}}
\hline
\makecell{Known mean\\ unknown covariance}& \makecell{$\delta_0=0.9$}\\
\hline
\makecell{Unknown mean\\ unknown covariance} & \makecell{$\delta_1=\delta_2=1$}\\
\hline
\makecell{$\phi-$divergence}&\makecell{$\theta_\phi=0.01$}\\
\hline
\makecell{Wasserstein distance}&\makecell{$\theta_W=0.01$\\$H=1000$}\\
\hline
\end{tabular}}
\end{minipage}
\end{table}
 \VS{For $\phi$-divergence based uncertainty set, we take the nominal distribution $\nu$ as a normal distribution with mean $\mu_{\nu}=\mu$ and covariance matrix $\Sigma_{\nu}=\Sigma$  where $\mu$ and $\Sigma$ are defined by \eqref{mean} and \eqref{cov}, respectively. 
 For Wasserstein distance based uncertainty set, 
  we take the number of observations  $H=1000$. The scenarios $(\Tilde{\xi}_i)_{i=1}^H$ are randomly generated by the reference distribution $\nu$. We generate a standard Gaussian vector by the command "\textit{x=numpy.random.normal(0,1,20)}". Using vector $x$, we generate a Gaussian vector with $\mu_{\nu}$ and  $\Sigma_{\nu}$ by using $\Tilde{\xi}_i = Bx + \mu_{\nu}$, where $\mu_{\nu}$ and $\Sigma_{\nu}$ are the mean vector and the covariance matrix defined by \eqref{mean} and \eqref{cov}, respectively, and 
  $B$ is the Cholesky factorization of $\Sigma_{\nu}$. To get the Cholesky factorization of a matrix, we use the command "\textit{numpy.linalg.cholesky}". We summarize \AL{the} other parameters related to all the uncertainty sets in Table \ref{2}.
}

\begin{table}[ht]
\centering
\caption{Optimal policies of \emph{CCMDP} 
and \emph{DRCCMDP} with full and nonnegative supports}\label{table 0 of optimal policies}
\centering
\vspace{.1cm}
\scalebox{0.6}{\begin{tabular}{|*{11}{c|}}
\hline
\backslashbox{State(s) \kern-3em}{\kern-1em Optimal \\ policies} & \makecell{CCMDP \\Gaussian\\(p,1-p)} & \makecell{Full support\\known mean\\ known covariance\\(p,1-p)} & \makecell{Full support \\ known mean\\ unknown covariance\\(p,1-p)}& \makecell{Full support \\ unknown mean\\ unknown covariance\\(p,1-p)}& \makecell{$\phi-$divergence\\ (Modified $\chi^2$)\\(p,1-p)}&\makecell{$\phi-$divergence\\(variation)\\(p,1-p)} \\
\hline
1&(0,\,\,1) &(0,\,\,1) &(0,\,\,1) & (0,\,\,1) & (0,\,\,1)&(0,\,\,1)\\ 
\hline
2&(0,\,\,1) &(0,\,\,1) &(0,\,\,1) & (0,\,\,1)&(0,\,\,1)& (0,\,\,1)\\ 
\hline
3&(0,\,\,1)&(0,\,\,1) & (0,\,\,1) & (0,\,\,1) &(0,\,\,1)& (0,\,\,1)\\ 
\hline
4&(0,\,\,1)&(0,\,\,1) & (0,\,\,1) &(0,\,\,1) & (0,\,\,1) & (0,\,\,1)\\ 
\hline
5&(0,\,\,1)&(0,\,\,1) & (0,\,\,1)&(0,\,\,1) & (0,\,\,1) & (0,\,\,1)\\  
\hline
6&(0,\,\,1)& (0,\,\,1) &(0,\,\,1) & (0,\,\,1) & (0,\,\,1) & (0,\,\,1)\\ 
\hline
7&(0,\,\,1)&(0,\,\,1)& (0,\,\,1) & (0,\,\,1) & (0,\,\,1) & (0,\,\,1)\\ 
\hline
8&(0,\,\,1)& (0,\,\,1) &(0,\,\,1)& (0,\,\,1) & (0,\,\,1) & (0,\,\,1)\\ 
\hline
9 &(0,\,\,1)& (0.64,\,\,0.36) & (0.64,\,\,0.36)&
(0.6,\,\,0.4) & (0.27,\,\,0.73)& (0.05,\,\,0.95)\\ \hline
10& (0.9,\,\,0.1) & (0.91,\,\,0.09) &(0.91,\,\,0.09) &(0.91,\,\,0.09)& (0.9,\,\,0.1)& (0.9,\,\,0.1)\\
\hline
\end{tabular}}
\end{table}
\begin{table}[ht]
\centering
\caption{Optimal policies of \emph{CCMDP} 
and \emph{DRCCMDP} with full and nonnegative supports (continued)}\label{table 1 of optimal policies}
\vspace{.1cm}
\centering
\scalebox{0.6}{\begin{tabular}{|*{11}{c|}}
\hline
\makecell{$\phi-$divergence\\(Kullbach-Leibler)\\(p,1-p)} &\makecell{$\phi-$divergence\\(Hellinger )\\(p,1-p)} & \makecell{Full support\\Wasserstein\\(p,1-p)} & \makecell{Nonnegative \\known mean\\ known covariance\\(p,1-p)} & \makecell{Nonnegative \\known mean\\ unknown covariance\\(p,1-p)}& \makecell{Nonnegative \\unknown mean\\ unknown covariance\\(p,1-p)}& \makecell{Nonnegative \\Wasserstein\\(p,1-p)}\\
\hline
(0,\,\,1)& (0,\,\,1)& (0,\,\,1) & (0,\,\,1) &(0,\,\,1) & (0,\,\,1)& (0,\,\,1)\\ 
\hline
(0,\,\,1)& (0,\,\,1)& (0,\,\,1) &(0,\,\,1) &(0,\,\,1) & (0,\,\,1)& (0,\,\,1)\\ 
\hline
(0,\,\,1)& (0,\,\,1)& (0,\,\,1) &(0,\,\,1) & (0,\,\,1) & (0,\,\,1) & (0,\,\,1)\\ 
\hline
(0,\,\,1)& (0,\,\,1)& (0,\,\,1) &(0,\,\,1) & (0,\,\,1) & (0,\,\,1) & (0,\,\,1)\\ 
\hline
(0,\,\,1)& (0,\,\,1)& (0,\,\,1)  &(0,\,\,1) & (0,\,\,1) & (0,\,\,1) & (0,\,\,1)\\ 
\hline
(0,\,\,1)& (0,\,\,1)& (0,\,\,1) &(0,\,\,1) & (0,\,\,1) & (0,\,\,1) & (0,\,\,1)\\ 
\hline
(0,\,\,1)& (0,\,\,1)& (0,\,\,1) &(0,\,\,1) & (0,\,\,1) & (0,\,\,1) & (0,\,\,1)\\  
\hline
(0,\,\,1)& (0,\,\,1)& (0,\,\,1)  &(0,\,\,1) & (0,\,\,1) & (0,\,\,1) & (0,\,\,1)\\ 
\hline
(0.25,\,\,0.75)& (0.28,\,\,0.72)& (0.02,\,\,0.98) & (0.62,\,\,0.38) & (0.62,\,\,0.38)&
(0.59,\,\,0.41) & (0.01,\,\,0.99)\\ \hline
(0.9,\,\,0.1)& (0.9,\,\,0.1)& (0.9,\,\,0.1) & (0.91,\,\,0.09) &(0.91,\,\,0.09) &(0.91,\,\,0.09)& (0.9,\,\,0.1)\\ 
\hline
\end{tabular}}
\end{table}

\VS{We compute an optimal policy of the CCMDP problem \eqref{CCMDP-equi}, where $\hat{R}$ follows a normal distribution with mean vector and covariance matrix defined by 
\eqref{mean} and \eqref{cov}, by solving an equivalent SOCP problem \cite{delage2010percentile}. The optimal policies of the DRCCMDP problem for all the uncertainty sets are computed by solving  \AL{the} proposed equivalent optimization problems.
We present the optimal policies of CCMDP and  DRCCMDP with full support and nonnegative support in Tables \ref{table 0 of optimal policies} and \ref{table 1 of optimal policies}, where $p$ is the probability of "repair" action and $1-p$ is the probability of "do not repair" action.  It is clear from Tables \ref{table 0 of optimal policies} and \ref{table 1 of optimal policies} that the optimal repair policy corresponding to all the uncertainty sets for first eight states is same. At state 9 \AL{the} probability of repair is \AL{greater} than the probability of do not repair for moments based uncertainty sets whereas for statistical distance based uncertainty sets the probability of repair is less than the probability of do not repair. This shows that \AL{the} statistical distance based uncertainty sets give better optimal policy as compared to moments based uncertainty sets. At \AL{the} last state, \AL{the} optimal policy is to choose repair action with a very high probability for all the uncertainty sets.}

 We present the time analysis by considering the number of states for all uncertainty sets between 1000 and 10000. All \AL{the} parameters are taken similar to the case of 10 states. The results are presented in Figure \ref{CPU time} which shows that the CPU \AL{time is almost always the same }  to solve \NGU{SOCP \eqref{SOCP known known real full support} with $\kappa=\sqrt{\frac{1-\epsilon}{\epsilon}}$} and the MISOCP \eqref{reformulation Wasserstein p=1 data-driven big M-method} while \AL{additional CPU time is required} to solve the \AL{SDP relaxations of the } \NGU{copositive optimization problem \eqref{matrix form unbounded nonegative polytopic}} and the biconvex optimization problem \eqref{nonnegative wasserstein continuous optimization}.
\begin{figure}[ht]\label{CPU time}
\centering
\includegraphics[width=10cm]{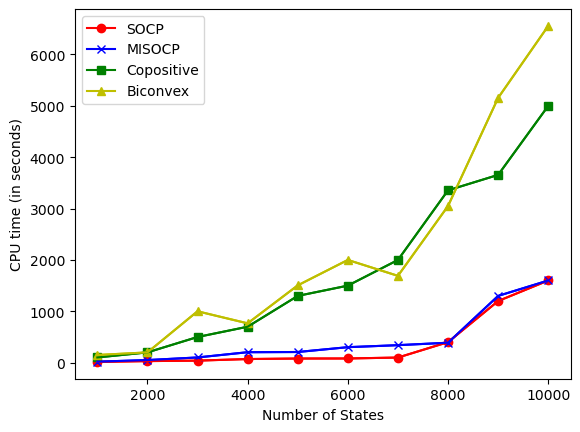}
\caption{CPU time (in seconds) to solve SOCP \eqref{SOCP known known real full support} with $\kappa=\sqrt{\frac{1-\epsilon}{\epsilon}}$, MISOCP \eqref{reformulation Wasserstein p=1 data-driven big M-method}, copositive optimization problem \eqref{matrix form unbounded nonegative polytopic} and biconvex optimization problem \eqref{nonnegative wasserstein continuous optimization} with different number of states.}
\end{figure}
\section{Conclusions}\label{conclusion}
\VS{
We study a DRCCMDP problem under various moments and statistical distance based uncertainty sets defined using $\phi$-divergence and Wasserstein distance metric. We propose equivalent SOCP, MISOCP, copositive \NGUNEW{optimization problem} and biconvex optimization problem, depending on the choice of the uncertainty set, for the DRCCMDP problem.
All these optimization problems except biconvex optimization problems and \NGUNEW{copositive optimization problems} can be solved efficiently using known optimization solvers. We perform numerical experiments, using the optimization solvers in python, on a machine replacement problem using randomly generated data. The numerical experiments are performed on the DRCCMDP problem up to 10000 states and it is very clear from our time analysis that these problems can be solved very efficiently.  
}

\section*{Acknowledgement}  This research was supported by DST/CEFIPRA Project No.
IFC/4117/DST-CNRS-5th call/2017-18/2 and
CNRS Project No. AR/SB:2018-07-440.



\appendix
\section{Proof of Lemma \ref{lemma strong duality nonnegative polytopic}}\label{Appendix Lemma non negative polytopic}

\hfill

Consider the optimization problem
\begin{align} \label{primal problem known known}
    & v_{\text{P}}(\mu,\Sigma) = \sup_{F \in \mathcal{C}^+} \int_{\varphi} \mathbf{1}_{\left\{\rho^{\text{T}} \hat{R} \leq y\right\}}\text{d}F(\hat{R}) \,\,\, \nonumber\\
    \mbox{s.t..}\,\,\, & (\text{i}) \quad  \int_{\varphi} \text{d}F(\hat{R}) = 1, \nonumber\\
    & (\text{ii})\quad \int_{\varphi} (\hat{R}-\mu)(\hat{R}-\mu)^{\text{T}} \text{d}F(\hat{R}) = \Sigma, \nonumber \\
    & (\text{iii})\quad \int_{\varphi} \hat{R} \text{d}F(\hat{R}) = \mu,
\end{align}
where $\mathcal{C}^+$ is the set of all positive measures on $\mathbb{R}_{+}^{|\mathcal{K}|}$. The dual problem of \eqref{primal problem known known} is given by
\begin{align} \label{dual problem known known}
    & v_{\text{D}}(\mu,\Sigma) = \inf\quad -t-Q \circ \Sigma - q^{\text{T}} \mu \,\,\, \nonumber\\
    \mbox{s.t..}\,\,\, & (\text{i}) ~ \mathbf{1}_{\left\{\rho^{\text{T}} \xi \leq y\right\}} + q^{\text{T}} \xi + \xi^{\text{T}} Q \xi - 2\xi^{\text{T}} Q \mu + \mu^{\text{T}} Q \mu + t \leq 0,\  \forall\ \xi \in \mathbb{R}_{+}^{|\mathcal{K}|},\nonumber\\
    & (\text{ii})   ~  Q\in \mathcal{S}^{|\mathcal{K}|},
\end{align}
\VS{where $t$, $q$, and $Q$ are the dual variables
associated with the constraints $(\text{i})$, $(\text{ii})$ and $(\text{iii})$ of \eqref{primal problem known known}, respectively.}
In Theorem 3.4 of \cite{cheng2014distributionally}, under the assumption $\mu \in \text{RI}(\varphi)$, the authors show that the Dirac distribution $\delta_\mu$ lies in the relative interior of the distributional \VS{uncertainty} set which implies that the weaker condition of Proposition 3.4 of \cite{shapiro2001duality} holds. However, it is not trivial to find a strictly feasible point inside our distributional \VS{uncertainty} set. \VS{Let $(t^*_j, Q^*_j, q^*_j)_{j \in \mathbb{N}}$ be a sequence of feasible solutions of \eqref{dual problem known known} such that }
\begin{align}\label{limit vd}
    -t^*_j-Q^*_j \circ \Sigma - q_j^{*\text{T}} \mu \rightarrow v_{\text{D}}(\mu,\Sigma), \ \text{as} \ j \rightarrow \infty.
\end{align}
For each $j \in \mathbb{N}$, let $r^*_j=\max(0,q^*_j) - q^*_j$, where $\max(0,q^*_j)$ denotes a $|\mathcal{K}|$-dimensional vector with $i^{\text{th}}$ component equal to the maximum value between $0$ and the $i^{\text{th}}$ component of $q^*_j$, for every $i = 1,\ldots,|\mathcal{K}|$. Let $\epsilon_j$ be a strictly positive decreasing sequence such that $\epsilon_j r^*_j\rightarrow 0$ componentwise and $\epsilon_j\rightarrow 0$, when $j\rightarrow \infty$. Let $x_j = \epsilon_j \NGUNEW{\mathbf{1}}$, where $\NGUNEW{\mathbf{1}}$ denotes the vector with all components equal to $1$. Then, 
 $   r_j^{*\text{T}} x_j \rightarrow 0$ as $j \rightarrow \infty$.
For each $j \in \mathbb{N}$, consider the following conic optimization problem
\begin{align} \label{alter primal problem known known}
    & v^j_{\text{P}}(\mu,\Sigma) = \sup_{F \in \mathcal{C}^+} \int_{\varphi} \mathbf{1}_{\left\{\rho^{\text{T}} R \leq y\right\}}\text{d}F(R) \,\,\, \nonumber \\
    \mbox{s.t.}\,\,\, & (\text{i}) \quad \int_{\varphi} \text{d}F(R) = 1, \nonumber\\
    & (\text{ii})\quad \int_{\varphi} (R-\mu)(R-\mu)^{\text{T}} \text{d}F(R) = \Sigma, \nonumber \\
    & (\text{iii})\quad \mu \leq \int_{\varphi} R \text{d}F(R) \leq \mu+x_j.
\end{align}
The dual problem of \eqref{alter primal problem known known} is given by
\begin{align} \label{alter dual problem known known}
    & v^j_{\text{D}}(\mu,\Sigma) = \inf\quad -t-Q \circ \Sigma + (r-h)^{\text{T}} \mu + r^T x_j \,\,\, \nonumber\\
    \mbox{s.t.}\,\,\, & (\text{i}) \quad  \mathbf{1}_{\left\{\rho^{\text{T}} \xi \leq y\right\}} + (h-r)^{\text{T}} \xi + \xi^{\text{T}} Q \xi - 2\xi^{\text{T}} Q \mu + \mu^{\text{T}} Q \mu + t \leq 0,\ \forall \ \xi \in \mathbb{R}_{+}^{|\mathcal{K}|},\nonumber\\
    & (\text{ii}) \quad  h, r \in \mathbb{R}_+^{|\mathcal{K}|}, Q\in \mathcal{S}^{|\mathcal{K}|},
\end{align}
\VS{where $t$, $Q$, $r$ and $h$ are the dual variables of the constraint $(\text{i})$, $(\text{ii})$ and $(\text{iii})$ of \eqref{alter primal problem known known}, respectively. The vector $(t,Q,h,r)$ such that 
 $t=t_j^*, Q=Q_j^*$, $h=\max(0,q_j^*)$, $r=r_j^*$   is a feasible solution of \eqref{alter dual problem known known}.} Hence, 
\begin{align}\label{left side condition}
    &v^j_{\text{D}}(\mu,\Sigma) \leq -t^*_j-Q^*_j \circ \Sigma - q_j^{*\text{T}} \mu + r_j^{*\text{T}} x_j, \ \forall \ j\in \mathbb{N}.
\end{align}
Since the feasibility set of \eqref{definition vP vP} is non-empty, there exists a nonnegative distribution $F^*$ such that $\mathbb{E}(F^*)=\mu$ and $ \text{Cov}(F^*)=\Sigma$. Let $F_j$ be a distribution with support

\noindent$\varphi_j:= \left\{\xi\,\|\,\,\xi \in \mathbb{R}_+^\mathcal{K}, \xi \geq \frac{x_j}{2},\ \mbox{componentwise}\right\}$,  defined by
\begin{align*}
    F^*(\xi) = F_j(\xi+\frac{x_j}{2}), \ \forall\ \xi \in \mathbb{R}_+^\mathcal{K}.
\end{align*}
It is clear that $F_j$ is a feasible \VS{solution} of \eqref{alter primal problem known known} and $\varphi_j \subset \text{RI}(\varphi)$. Hence, $F_j$ belongs to the relative interior of the distributional uncertainty set \VS{which implies that strong duality holds, i.e., $v_\text{P}^j(\mu,\Sigma) = v_\text{D}^j(\mu,\Sigma)$ for all $j\in \mathbb{N}$. Since the objective function of \eqref{alter primal problem known known} is a continuous function of $F$ and $x_j\rightarrow 0$ as $j\rightarrow \infty$, then $v_\text{P}^j(\mu,\Sigma)\rightarrow v_\text{P}(\mu,\Sigma)$ as $j\rightarrow \infty$. Therefore, it is  sufficient to prove that $v_\text{D}^j(\mu,\Sigma) \rightarrow v_\text{D}(\mu,\Sigma)$ as $j\rightarrow \infty$. }
\VS{
It is clear that the feasible  sets of \eqref{alter dual problem known known} and \eqref{dual problem known known} are equivalent and objective function of \eqref{alter dual problem known known} has additional positive term. Therefore,}
\begin{align}\label{right side condition}
v_\text{D}^j(\mu,\Sigma) \geq v_\text{D}(\mu,\Sigma), \ \forall \ j\in \mathbb{N}.
\end{align}
\VS{Using \eqref{limit vd},  \eqref{left side condition} and \eqref{right side condition} and the fact that $r_j^{*T}x_j \rightarrow 0$ as $j\rightarrow \infty$, we have
$v_\text{D}^j(\mu,\Sigma) \rightarrow v_\text{D}(\mu,\Sigma)$ as $j\rightarrow \infty$.}
\section{ Proof of Theorem \ref{theorem phi divergencess} - Case Hellinger distance}\label{app_Hellinger}

From Table \ref{table of phi divergences}, the conjugate of $\phi$ has the following form
\begin{equation}\label{form of Hellinger conjugate}
    \phi^*(r)=
    \begin{cases}
        & \frac{r}{1-r},\quad \mbox{if}\ r<1,\\
        &\infty,\quad \mbox{if}\ r \ge 1.
    \end{cases}
\end{equation}
Let
\begin{align}\label{Lagrangian dual}
    &L=\sup_{\lambda > 0, \beta \in \mathbb{R}}\left\{\beta-\lambda \theta_\phi - \lambda \phi^*\left(\frac{-1+\beta}{\lambda}\right)\mathbb{P}_{\nu}(O)- \lambda \phi^*\left(\frac{\beta}{\lambda}\right)(1-\mathbb{P}_{\nu}(O))\right\}.
\end{align}
The constraint $(\text{i})$ of \eqref{DRCCMDP-opt-ex} is equivalent to
\begin{align}\label{constraint Hellinger}
    L\geq 1-\epsilon.
\end{align}
We consider two cases as follows:\\
\noindent    \textbf{Case 1:} Let $\frac{\beta}{\lambda}<1$. 
    Since $\lambda>0$, the following inequality holds
    \begin{align*}
        \frac{\beta-1}{\lambda} < \frac{\beta}{\lambda} < 1.
    \end{align*}
    From \eqref{form of Hellinger conjugate}, we have 
    \[
      \phi^*\left(\frac{\beta}{\lambda}\right) =  \frac{\NGU{\beta}}{\lambda-\beta}, \  \phi^*\left(\frac{\beta-1}{\lambda}\right)=\frac{\NGU{\beta-1}}{\lambda+1-\beta}.
    \]
    \VS{Consequently, it follows from \eqref{Lagrangian dual} 
    that 
    \[
    L=\sup_{\lambda > 0,  \beta<\lambda}\left\{\mathbb{P}_{\nu}(O)\frac{\lambda^2}{(\lambda-\beta)(\lambda-\beta+1)}-\frac{\beta^2}{\lambda-\beta} - \lambda \theta_\phi\right\}.
    \]
    Let $\eta=\lambda-\beta$. Then, we can write
    \[
    L=\sup_{\lambda > 0, \eta>0}\left\{\lambda^2\left(\frac{\mathbb{P}_{\nu}(O)}{\eta(\eta+1)}-\frac{1}{\eta}\right)+\lambda(2-\theta_\phi)-\eta\right\}.
    \]
    Let $g(\lambda,\eta)=\lambda^2\left(\frac{\mathbb{P}_{\nu}(O)}{\eta(\eta+1)}-\frac{1}{\eta}\right)+\lambda(2-\theta_\phi)-\eta$. 
     It is a second-order polynomial of $\lambda$ and the coefficient of $\lambda^2$ is negative because $0\leq\mathbb{P}_{\nu}(O)\leq 1$ and $\eta>0$. It is well known that the maximum value of a second order polynomial $f(x)=ax^2+bx+c$ with \NGU{$a<0$} is $c-\frac{b^2}{4a}$ and it holds at $x=\frac{-b}{2a}$. Hence, the maximum value of $g(\lambda,\eta)$ holds at $\lambda^*=\frac{\eta(\eta+1)(2-\theta_\phi)}{2(1+\eta-\mathbb{P}_{\nu}(O))}$.
     Since $\theta_\phi<2$,  $\lambda^* > 0$. Therefore, for a given $\eta>0$, the optimal value $L$ holds at $\lambda^*$ and $L = c-\frac{b^2}{4a}$, where $c=-\eta$, $b=2-\theta_{\phi}$, $a=\frac{\mathbb{P}_{\nu}(O)}{\eta(\eta+1)}-\frac{1}{\eta}$, which implies that
    \begin{align}\label{mnlb}
    L&= \sup_{\eta>0}\left\{-\eta+\frac{(2-\theta_\phi)^2\eta(\eta+1)}{4(\eta+1-\mathbb{P}_{\nu}(O))}\right\}.
    \end{align}
    Let $u= \eta+1-\mathbb{P}_{\nu}(O)$, then $\eta>0$ is equivalent to $u > 1-\mathbb{P}_{\nu}(O)$ and we can write 
    \begin{align*}
       L &=\sup_{u > 1-\mathbb{P}_{\nu}(O)}\Bigg\{\left(\frac{(2-\theta_\phi)^2}{4}-1\right)u+\frac{(2-\theta_\phi)^2\mathbb{P}_{\nu}(O)(\mathbb{P}_{\nu}(O)-1)}{4}\frac{1}{u}\\
        &+ 1 - \mathbb{P}_{\nu}(O) + \frac{(2-\theta_\phi)^2(2\mathbb{P}_{\nu}(O)-1)}{4}\Bigg\},\\
        &= \sup_{u > 1-\mathbb{P}_{\nu}(O)}\NGU{G(u)},
    \end{align*}
    where $G(u)=a_1u+\frac{b_1}{u}+c_1$ such that
     \begin{align*}
     &a_1=\frac{(2-\theta_\phi)^2}{4}-1, \ 
b_1=\frac{(2-\theta_\phi)^2\mathbb{P}_{\nu}(O)(\mathbb{P}_{\nu}(O)-1)}{4},\\
     &c_1=1 - \mathbb{P}_{\nu}(O) + \frac{(2-\theta_\phi)^2(2\mathbb{P}_{\nu}(O)-1)}{4}.
     \end{align*}
     }
     Since $0<\theta_{\phi}<2$ and $0 \leq \mathbb{P}_{\nu}(O) \leq 1$,  $a_1<0$ and $b_1\leq 0$. It is clear that $G$ is decreasing on $(u^*,\infty)$, increasing on $(-u^*,u^*)$ and decreasing on $(-\infty,-u^*)$,  where
     \begin{align}\label{formula of gamma}
        &u^* = \sqrt{\frac{b_1}{a_1}}=\sqrt{\frac{(2-\theta_\phi)^2}{4-(2-\theta_\phi)^2}\mathbb{P}_{\nu}(O)(1-\mathbb{P}_{\nu}(O))},\\
        &G(u^*)=a_1u^*+\frac{b_1}{u^*}+c_1=-2\sqrt{a_1b_1}+c_1.\nonumber
    \end{align}
     \NGU{If $u^*\leq 1-\mathbb{P}_{\nu}(O)$, we deduce that $(1-\mathbb{P}_{\nu}(O),\infty) \subset (u^*,\infty)$. Since $G$ is decreasing on $(u^*,\infty)$, it implies that $G$ is decreasing on $(1-\mathbb{P}_{\nu}(O),\infty)$. \VS{Hence, the optimal value of $G$ is attained when $u =
     1-\mathbb{P}_{\nu}(O)$, i.e, $\eta=0$}. From \eqref{mnlb}, $L=0$ which violates the constraint \eqref{constraint Hellinger}. Therefore, $u^* > 1-\mathbb{P}_{\nu}(O) > 0$}. Since, $G$ is decreasing on $(u^*,\infty)$ and increasing on $(1-\mathbb{P}_{\nu}(O),u^*)$, then $u=u^*$ is the optimal solution of $G(u)$ and $L=-2\sqrt{a_1b_1}+c_1$. Therefore, 
     \begin{align*}
     L=&\NGU{-}2\sqrt{\frac{(2-\theta_\phi)^2}{4}\left(1-\frac{(2-\theta_\phi)^2}{4}\right)\mathbb{P}_{\nu}(O)(1-\mathbb{P}_{\nu}(O))}\\
     &+1 - \mathbb{P}_{\nu}(O) + \frac{(2-\theta_\phi)^2(2\mathbb{P}_{\nu}(O)-1)}{4}.
    \end{align*}
    Then, \eqref{constraint Hellinger} is rewritten equivalently as follows
    \begin{align}\label{final formula of vD}
     &\NGU{-}2\sqrt{\frac{(2-\theta_\phi)^2}{4}\left(1-\frac{(2-\theta_\phi)^2}{4}\right)\mathbb{P}_{\nu}(O)(1-\mathbb{P}_{\nu}(O))}\nonumber \\
    &\geq 
    \left(1-\frac{(2-\theta_\phi)^2}{2}\right)\mathbb{P}_{\nu}(O)+\frac{(2-\theta_\phi)^2}{4} - \epsilon.
    \end{align}
    By taking the square on both side of \eqref{final formula of vD},  we get 
    \begin{align}\label{gggg}
     &(2-\theta_\phi)^2\left(1-\frac{(2-\theta_\phi)^2}{4}\right)\mathbb{P}_{\nu}(O)(1-\mathbb{P}_{\nu}(O))\nonumber \\
&    \leq 
    \Bigg[\left(1-\frac{(2-\theta_\phi)^2}{2}\right)\mathbb{P}_{\nu}(O)+\frac{(2-\theta_\phi)^2}{4} - \epsilon\Bigg]^2.
    \end{align}
    \VS{By rewriting \eqref{gggg}, we get the following second-order inequality in $\mathbb{P}_{\nu}(O)$
    \[
    \big(\mathbb{P}_{\nu}(O)\big)^2 + B\ \mathbb{P}_{\nu}(O) + C \geq 0,
    \]
    which is equivalent to 
    \begin{align}\label{3 condition}
     \big(\mathbb{P}_{\nu}(O)-x_{\text{max}}\big)\big(\mathbb{P}_{\nu}(O)-x_{\text{min}}\big) \geq 0,
    \end{align}
    }
    where $x_{\text{max}} = \frac{-B+\sqrt{\Delta}}{2}$, $x_{\text{min}} = \frac{-B-\sqrt{\Delta}}{2}$ and $B,C,\Delta$ are given in Table \ref{table of function f}. It is clear that \eqref{final formula of vD} is equivalent to either $\mathbb{P}_{\nu}(O) \geq x_{\text{max}}$ or $\mathbb{P}_{\nu}(O) \leq x_{\text{min}}$. Moreover, $x_{\text{max}}$ and $x_{\text{min}}$ are  solutions of the following two equalities
    \begin{equation}\label{for xmax}
     \NGU{-}2\sqrt{\frac{(2-\theta_\phi)^2}{4}\left(1-\frac{(2-\theta_\phi)^2}{4}\right)x(1-x)}
    = 
    \left(1-\frac{(2-\theta_\phi)^2}{2}\right)x+\frac{(2-\theta_\phi)^2}{4} - \epsilon,
    \end{equation}
    and
    \begin{equation}\label{oaoa}
     2\sqrt{\frac{(2-\theta_\phi)^2}{4}\left(1-\frac{(2-\theta_\phi)^2}{4}\right)x(1-x)} 
    =
    \left(1-\frac{(2-\theta_\phi)^2}{2}\right)x+\frac{(2-\theta_\phi)^2}{4} - \epsilon.
    \end{equation}
    Since $\theta_\phi< 2-\sqrt{2}$, we deduce that $1-\frac{(2-\theta_\phi)^2}{2}<0$. Therefore, we have
    \begin{equation*}
    \left(1-\frac{(2-\theta_\phi)^2}{2}\right)x_{\text{min}}+\frac{(2-\theta_\phi)^2}{4} - \epsilon
    > \left(1-\frac{(2-\theta_\phi)^2}{2}\right)x_{\text{max}}+\frac{(2-\theta_\phi)^2}{4} - \epsilon,
    \end{equation*}
    which implies that $x_{\text{max}}$ is a solution of \eqref{for xmax} and $x_{\text{min}}$ is a solution of \eqref{oaoa}. Hence, the condition $\mathbb{P}_{\nu}(O) \leq x_{\text{min}}$ implies that
    \begin{equation*}
     \left(1-\frac{(2-\theta_\phi)^2}{2}\right)\mathbb{P}_{\nu}(O)+\frac{(2-\theta_\phi)^2}{4} - \epsilon
     \ge \left(1-\frac{(2-\theta_\phi)^2}{2}\right)x_{\text{min}}+\frac{(2-\theta_\phi)^2}{4} - \epsilon>0,
    \end{equation*}
    which violates the constraint \eqref{final formula of vD}. Then, \eqref{final formula of vD} is equivalent to $\mathbb{P}_{\nu}(O) \geq x_{\text{max}}$, \VS{ i.e., the constraint $(\text{i})$ of \eqref{DRCCMDP-opt-ex} is equivalent to
    \[
    \mathbb{P}_{\nu}(\rho^T\hat{R}\ge y) \ge \frac{-B+\sqrt{\Delta}}{2}.
    \]
    }
     \textbf{Case 2:} \VS{ Let $1 \leq \frac{\beta}{\lambda}$. 
        From \eqref{form of Hellinger conjugate}, 
         $\phi^*\left(\frac{\beta}{\lambda}\right) = \infty$, which in turn implies that       $L=-\infty$ and it violates the constraint \eqref{constraint Hellinger}.}
         
\section{Proof of Lemma \ref{transformation lemma full support}} \label{app_C}

For each $i = 1,\ldots,H$, we consider two cases as follows:

\noindent \textbf{Case 1:} Let $\rho^{\text{T}} \Tilde{\xi_i} \leq y$.
In this case, it is clear that $\inf_{\rho^{\text{T}}z \leq y}||\Tilde{\xi_i}-z||_2=0$ and the optimal value holds at $z=\Tilde{\xi_i}$.

\noindent \textbf{Case 2:} Let $\rho^{\text{T}} \Tilde{\xi_i} > y$.
Geometrically, the term $\inf_{\rho^{\text{T}}z \leq y}||\Tilde{\xi_i}-z||_2$ can be interpreted as the distance between $\Tilde{\xi_i}$ and the hyper plane $\left\{z\,\,|\,\, \rho^{\text{T}} z = y\right\}$.  Assume that the optimal value of $\inf_{\rho^{\text{T}}z \leq y}||\Tilde{\xi_i}-z||_2$ holds at $z=z^*$. If $\rho^{\text{T}} z^* < y$, since $\rho^{\text{T}} \Tilde{\xi_i} > y$, we deduce that there exists $z_0$ on $\text{Seg}(z^*,\Tilde{\xi_i})$ such that $\rho^{\text{T}} z_0 = y$, where $\text{Seg}(z^*,\Tilde{\xi_i}) := \left\{z\,\,|\,\,z = z^* + t(\Tilde{\xi_i}-z^*),\quad 0<t<1\right\}$. It is clear that $||\Tilde{\xi_i}-z^*||_2 > ||\Tilde{\xi_i}-z_0||_2$. However, $||\Tilde{\xi_i}-z^*||_2=\inf_{\rho^{\text{T}}z \leq y}||\Tilde{\xi_i}-z||_2$, which gives a contradiction. Therefore, $\rho^{\text{T}} z^* = y$. \VS{We can write $\inf_{\rho^{\text{T}}z \leq y}||\Tilde{\xi_i}-z||_2$ equivalently as}
\begin{align}\label{in-bet-app}
    & \inf{||\Tilde{\xi_i}-z||_2} \nonumber\\
    \mbox{s.t.}\quad \quad &\rho^{\text{T}}z=y.
\end{align}
\VS{Using the KKT conditions, the optimal solution of \eqref{in-bet-app} satisfies} 
\begin{align}\label{eq111}
     2(\Tilde{\xi_i}-z^*)-\lambda \rho = 0,
\end{align}
where $\lambda$ is the Lagrange multiplier associated with the equality constraint.
By taking the inner product of \eqref{eq111} with $\rho$, we have
\begin{align*}
   2(\Tilde{\xi_i}-z^*)^T \rho - \lambda ||\rho||^2_2 = 0,
\end{align*}
which implies that
\begin{align}\label{formule lambda} \lambda = \frac{2(\Tilde{\xi_i}-z^*)^T \rho}{||\rho||^2_2}.
\end{align}
On the other hand, by taking inner product of \eqref{eq111} with $\Tilde{\xi_i}-z^*$, we get
\begin{align}\label{rrtt}
    &2||\Tilde{\xi_i}-z^*)||_2^2  - \lambda \rho^T(\Tilde{\xi_i}-z^*)=0.
\end{align}
\VS{Using \eqref{formule lambda}, \eqref{rrtt} and $\rho^Tz^*=y$, we have}
\begin{align*}
 ||\Tilde{\xi_i}-z^*)||_2 = \frac{\rho^T\Tilde{\xi_i}-y}{||\rho||_2}.
\end{align*}

\section{Proof of Lemma \ref{sufficient condition to apply big-M method}} \label{app_D}

Let $(y,\rho)$ be a feasible solution of \eqref{extra deterministic wasserstein full support} which implies that the constraint $(\text{i})$ of \eqref{DRCCMDP-opt} holds. \VS{Since, reference distribution $\nu$ always belongs to uncertainty set \eqref{uncertainty Wasserstein set}, we have}
\begin{align}\label{gjgj}
    \frac{1}{H}\sum_{i=1}^{H}\mathbf{1}_{\left\{\rho^{\text{T}}\Tilde{\xi_i} \leq y\right\}} = \mathbb{P}_{\nu}\left(\rho^{\text{T}} \hat{R} \leq y \right) \leq \epsilon.
\end{align}
 It follows from \eqref{gjgj} that there exists $\Tilde{\xi}_i$ such that $\rho^{\text{T}}\Tilde{\xi_i}>y$ which in turn implies that
\begin{align}\label{upper bound y}
    & y < \max_{i = 1,\ldots,H}(\rho^{\text{T}}\Tilde{\xi_i}) < \max_{i = 1,\ldots,H}|\rho^{\text{T}}\Tilde{\xi_i}| + \frac{\theta_W}{\epsilon}||\rho||_2.
\end{align}
Moreover, from the constraint $(\text{iii})$ of \eqref{extra deterministic wasserstein full support}, we have
\begin{align}\label{lower bound y}
    &y \geq \min_{i = 1,\ldots,H}(\rho^{\text{T}}\Tilde{\xi_i})-\frac{\theta_W}{\epsilon}||\rho||_2 \geq - \max_{i = 1,\ldots,H}|\rho^{\text{T}}\Tilde{\xi_i}|-\frac{\theta_W}{\epsilon}||\rho||_2.
\end{align}
Using \eqref{lower bound y} and \eqref{upper bound y}, we get the following inequality
\begin{align}\label{in-bet-appC}
    |y| + |\rho^{\text{T}}\Tilde{\xi_i}| \leq 2 \max_{i = 1,\ldots,H}|\rho^{\text{T}}\Tilde{\xi_i}| + \frac{\theta_W}{\epsilon}||\rho||_2,\ \forall \ i=1,\ldots,H.
\end{align}
\VS{
Using \eqref{in-bet-appC}, Cauchy-Schwartz inequality, and the fact that $\rho$ is a probability measure, we have
\[
|y-\rho^{\text{T}} \Tilde{\xi_i}|\le M.
\]
}

\section{ Proof of Lemma \ref{strong duality nonnegative wasserstein}} \label{app_E}

The optimization problem $\inf_{z \in \mathbb{R}_+^{|\mathcal{K}|}, \rho^{\text{T}}z \leq y}||\Tilde{\xi_i}-z||_2$ can be reformulated as  following SOCP problem
\begin{align}\label{vp nonnegative wasserstein}
    &\min\quad t \nonumber\\
     \mbox{s.t.} \,\,\,\,\,
    & (\text{i}) \quad \rho^{\text{T}} z \leq y,\nonumber\\
    & (\text{ii}) \quad t \geq ||\Tilde{\xi_i}-z||_2,\nonumber\\
    & (\text{iii}) \quad z \in \mathbb{R}_+^{|\mathcal{K}|}.
\end{align}
\VS{
The Lagrangian dual problem of \eqref{vp nonnegative wasserstein} is given by
\begin{align*}
     \max_{\lambda_i \geq 0, \zeta_i \in \mathbb{R}_+^{|\mathcal{K}|}, \beta \geq 0} \quad \min_{t \in \mathbb{R}, z \in \mathbb{R}^{|\mathcal{K}|}} \mathcal{L}(t,\rho,z, \lambda_i, \beta, \zeta_i),
\end{align*}
where $\mathcal{L}(t,z, \lambda_i, \beta, \zeta_i)= t + \lambda_i(\rho^{\text{T}}z-y) - \zeta_i^{\text{T}} z + \beta (||\Tilde{\xi_i}-z||_2-t)$ such that $\lambda_i$, $\beta$ and $\zeta_i$ are the Lagrange multipliers associated with constraints  $(\text{i})$, $(\text{ii})$ and $(\text{iii})$ of \eqref{vp nonnegative wasserstein}, respectively.}
The inner minimization problem can be written as 
\begin{equation}\label{oooo}
    J(\lambda_i,\zeta_i,\beta)
    = \min_{t \in \mathbb{R},z \in \mathbb{R}^{|\mathcal{K}|}} \left\{t(1-\beta) + \beta ||\Tilde{\xi_i}-z||_2+\lambda _i\rho^{\text{T}} z - \zeta_i^{\text{T}} z - \lambda_i y\right\}.
\end{equation}
\VS{It is easy to see that $J(\lambda_i,\zeta_i,\beta)=-\infty$ if $\beta \neq 1$ and
it implies that the dual objective function value is $-\infty$. By using the strong duality of a primal-dual pair of SOCPs, the objective function value of primal problem is $-\infty$, i.e., $\inf_{z \in \mathbb{R}_+^{|\mathcal{K}|}, \rho^{\text{T}}z \leq y}||\Tilde{\xi_i}-z||_2=-\infty$ which is a contradiction.} Therefore, $\beta=1$ and the dual problem of \eqref{vp nonnegative wasserstein} is given by
\begin{align*}
   \max_{\lambda_i \geq 0, \zeta_i \in \mathbb{R}_+^{|\mathcal{K}|}} J(\lambda_i,\zeta_i,1).
\end{align*}
Using a change of variable $z_1= \Tilde{\xi_i}-z$, we have
\begin{align*}
    J(\lambda_i,\zeta_i,1) = \min_{z_1 \in \mathbb{R}^{|\mathcal{K}|}} \left\{||z_1||_2+(\zeta_i-\lambda_i \rho)^{\text{T}} z_1\right\} + \lambda_i( \rho^{\text{T}}\Tilde{\xi_i}-y)-\zeta_i^{\text{T}}\Tilde{\xi_i}.
\end{align*}
\VS{The above minimization problem is unbounded unless $||\zeta_i-\lambda_i \rho||_2 \le 1$ and it leads to the following dual problem of \eqref{vp nonnegative wasserstein}.}
\begin{align}\label{vD nonnegative wasserstein}
    &\max\quad \lambda_i(\rho^{\text{T}}\Tilde{\xi_i}-y) - \zeta_i^{\text{T}}\Tilde{\xi_i} \nonumber\\
     \mbox{s.t.} \,\,\,\,\,
    & (\text{i}) \quad ||\zeta_i-\lambda_i \rho||_2 \leq 1, \nonumber\\
    & (\text{ii}) \quad \lambda_i \geq 0, \zeta_i \in \mathbb{R}_+^{|\mathcal{K}|}.
\end{align}

\nocite{*}
\bibliographystyle{acm}
\bibliography{main.bib}

\begin{thebibliography}{10}

\bibitem{altman1999constrained}
{\sc Altman, E.}
\newblock {\em Constrained Markov Decision Processes}.
\newblock Routledge, New York, 1999.

\bibitem{ben2013robust}
{\sc Ben-Tal, A., Den~Hertog, D., De~Waegenaere, A., Melenberg, B., and Rennen,
  G.}
\newblock Robust solutions of optimization problems affected by uncertain
  probabilities.
\newblock {\em Management Science 59}, 2 (2013), 341--357.

\bibitem{bomze2017fresh}
{\sc Bomze, I.~M., Cheng, J., Dickinson, P.~J., and Lisser, A.}
\newblock A fresh cp look at mixed-binary qps: new formulations and
  relaxations.
\newblock {\em Mathematical Programming 166}, 1-2 (2017), 159--184.

\bibitem{bomze2019notoriously}
{\sc Bomze, I.~M., Cheng, J., Dickinson, P.~J., Lisser, A., and Liu, J.}
\newblock Notoriously hard (mixed-)binary qps: empirical evidence on new
  completely positive approaches.
\newblock {\em Computational Management Science 16}, 4 (2019), 593--619.

\bibitem{ElGhaouiJOTA}
{\sc Calafiore, G.~C., and El~Ghaoui, L.}
\newblock On distributionally robust chance-constrained linear programs.
\newblock {\em Journal of Optimization Theory and Applications 130}, 1 (2006),
  1--22.

\bibitem{cheng2014distributionally}
{\sc Cheng, J., Delage, E., and Lisser, A.}
\newblock Distributionally robust stochastic knapsack problem.
\newblock {\em SIAM Journal on Optimization 24}, 3 (2014), 1485--1506.

\bibitem{delage2010percentile}
{\sc Delage, E., and Mannor, S.}
\newblock Percentile optimization for {M}arkov decision processes with
  parameter uncertainty.
\newblock {\em Operations Research 58}, 1 (2010), 203--213.

\bibitem{delage2010distributionally}
{\sc Delage, E., and Ye, Y.}
\newblock Distributionally robust optimization under moment uncertainty with
  application to data-driven problems.
\newblock {\em Operations Research 58}, 3 (2010), 595--612.

\bibitem{esfahani2018data}
{\sc Esfahani, P.~M., and Kuhn, D.}
\newblock Data-driven distributionally robust optimization using the
  {W}asserstein metric: Performance guarantees and tractable reformulations.
\newblock {\em Mathematical Programming 171}, 1 (2018), 115--166.

\bibitem{gao2016distributionally}
{\sc Gao, R., and Kleywegt, A.~J.}
\newblock Distributionally robust stochastic optimization with {W}asserstein
  distance.
\newblock {\em Mathematics of Operations Research\/} (2022).

\bibitem{ghaoui2003worst}
{\sc Ghaoui, L.~E., Oks, M., and Oustry, F.}
\newblock Worst-case value-at-risk and robust portfolio optimization: A conic
  programming approach.
\newblock {\em Operations Research 51}, 4 (2003), 543--556.

\bibitem{givan2000bounded}
{\sc Givan, R., Leach, S., and Dean, T.}
\newblock Bounded-parameter {M}arkov decision processes.
\newblock {\em Artificial Intelligence 122}, 1-2 (2000), 71--109.

\bibitem{hanasusanto2018conic}
{\sc Hanasusanto, G.~A., and Kuhn, D.}
\newblock Conic programming reformulations of two-stage distributionally robust
  linear programs over {W}asserstein balls.
\newblock {\em Operations Research 66}, 3 (2018), 849--869.

\bibitem{Henrionpaper}
{\sc Henrion, R.}
\newblock Structural properties of linear probabilistic constraints.
\newblock {\em Optimization 56}, 4 (2007), 425--440.

\bibitem{hiriart2010variational}
{\sc Hiriart-Urruty, J.-B., and Seeger, A.}
\newblock A variational approach to copositive matrices.
\newblock {\em SIAM Review 52}, 4 (2010), 593--629.

\bibitem{iyengar2005robust}
{\sc Iyengar, G.~N.}
\newblock Robust dynamic programming.
\newblock {\em Mathematics of Operations Research 30}, 2 (2005), 257--280.

\bibitem{jiang2016data}
{\sc Jiang, R., and Guan, Y.}
\newblock Data-driven chance constrained stochastic program.
\newblock {\em Mathematical Programming 158}, 1 (2016), 291--327.

\bibitem{liu2022distributionally}
{\sc Liu, J., Lisser, A., and Chen, Z.}
\newblock Distributionally robust chance constrained geometric optimization.
\newblock {\em Mathematics of Operations Research 47}, 4 (2022), 2950--2988.

\bibitem{mannor2007bias}
{\sc Mannor, S., Simester, D., Sun, P., and Tsitsiklis, J.~N.}
\newblock Bias and variance approximation in value function estimates.
\newblock {\em Management Science 53}, 2 (2007), 308--322.

\bibitem{HNConference2022}
{\sc Nguyen, H.~N., Singh, V.~V., Lisser, A., and Arora, M.}
\newblock Zero-sum games with distributionally robust chance constraints.
\newblock In {\em 17th International Conference on Internet and Web
  Applications and Services (ICIW)\/} (2022), pp.~7--12.

\bibitem{nilim2005robust}
{\sc Nilim, A., and El~Ghaoui, L.}
\newblock Robust control of {M}arkov decision processes with uncertain
  transition matrices.
\newblock {\em Operations Research 53}, 5 (2005), 780--798.

\bibitem{pardo2018statistical}
{\sc Pardo, L.}
\newblock {\em Statistical Inference Based on Divergence Measures}.
\newblock Chapman and Hall/CRC Press, New York, 2018.

\bibitem{popescu2007robust}
{\sc Popescu, I.}
\newblock Robust mean-covariance solutions for stochastic optimization.
\newblock {\em Operations Research 55}, 1 (2007), 98--112.

\bibitem{puterman1994discrete}
{\sc Puterman, M., et~al.}
\newblock {\em Markov Decision Processes}.
\newblock John Wiley \& Sons, Inc., New York, 1994.

\bibitem{shapiro2001duality}
{\sc Shapiro, A.}
\newblock On duality theory of conic linear problems.
\newblock In {\em Semi-Infinite Programming. Nonconvex Optimization and Its
  Applications}, M.~{\'A}. Goberna and M.~A. L{\'o}pez, Eds., vol.~57. Springer
  US, Boston, MA, 2001, pp.~135--165.

\bibitem{singh2017distributionally}
{\sc Singh, V.~V., Jouini, O., and Lisser, A.}
\newblock Distributionally robust chance-constrained games: {E}xistence and
  characterization of {N}ash equilibrium.
\newblock {\em Optimization Letters 11}, 7 (2017), 1385--1405.

\bibitem{sion1958general}
{\sc Sion, M.}
\newblock On general minimax theorems.
\newblock {\em Pacific Journal of Mathematics 8}, 1 (1958), 171--176.

\bibitem{varagapriya2022constrained}
{\sc Varagapriya, V., Singh, V.~V., and Lisser, A.}
\newblock Constrained {M}arkov decision processes with uncertain costs.
\newblock {\em Operations Research Letters 50\/} (2022), 218--223.

\bibitem{VagapriyaCMDP}
{\sc Varagapriya, V., Singh, V.~V., and Lisser, A.}
\newblock Joint chance-constrained {M}arkov decision processes.
\newblock {\em Annals of Operations Research\/} (2022).

\bibitem{villani2009optimal}
{\sc Villani, C.}
\newblock {\em Optimal Transport}.
\newblock Springer, Berlin, Heidelberg, 2009.

\bibitem{villani2021topics}
{\sc Villani, C.}
\newblock {\em Topics in Optimal Transportation}.
\newblock American Mathematical Society, Providence, 2021.

\bibitem{white1994markov}
{\sc White~III, C.~C., and Eldeib, H.~K.}
\newblock Markov decision processes with imprecise transition probabilities.
\newblock {\em Operations Research 42}, 4 (1994), 739--749.

\bibitem{wiesemann2013robust}
{\sc Wiesemann, W., Kuhn, D., and Rustem, B.}
\newblock Robust {M}arkov decision processes.
\newblock {\em Mathematics of Operations Research 38}, 1 (2012), 153--183.

\bibitem{xu2018copositive}
{\sc Xu, G., and Burer, S.}
\newblock A copositive approach for two-stage adjustable robust optimization
  with uncertain right-hand sides.
\newblock {\em Computational Optimization and Applications 70}, 1 (2018),
  33--59.

\bibitem{zhao2018data}
{\sc Zhao, C., and Guan, Y.}
\newblock Data-driven risk-averse stochastic optimization with {W}asserstein
  metric.
\newblock {\em Operations Research Letters 46}, 2 (2018), 262--267.

\end{thebibliography}

\end{document}